\documentclass{amsart}
\usepackage{hyperref,enumerate,xcolor}
\usepackage{amsmath,amsfonts,amsthm,amscd,mathabx}

\newcommand{\tg}{\widetilde{g}}
\newcommand{\bg}{\bar{g}}

\newcommand{\wg}{\widehat{g}}
\newcommand{\wP}{\widehat{P}}

\newcommand{\hg}{\hat{g}}
\newcommand{\wD}{\widehat{\nabla}}

\newcommand{\bm}{\overline{\mu}}
\newcommand{\bR}{\overline{R}}
\newcommand{\bW}{\overline{W}}
\newcommand{\hm}{\widehat{\mu}}

\newcommand{\wh}{\widetilde{h}}

\DeclareMathOperator{\tr}{tr}

\DeclareMathOperator{\grad}{grad}
\DeclareMathOperator{\vol}{Vol}

\newtheorem{theorem}{Theorem}[section]
\newtheorem{lemma}[theorem]{Lemma}
\newtheorem{claim}[theorem]{Claim}
\newtheorem{proposition}[theorem]{Proposition}
\newtheorem{corollary}[theorem]{Corollary}
\newtheorem{question}[theorem]{Question}

\theoremstyle{definition}
\newtheorem{definition}[theorem]{Definition}

\numberwithin{equation}{section}

\title[Self-Dual and Even Einstein Metrics]{Self-dual and even Poincar\'e-Einstein metrics in dimension four}

\author{Matthew J. Gursky}
\address{Department of Mathematics \\
         University of Notre Dame\\
         Notre Dame, IN 46556}
\email{\href{mgursky@nd.edu}{mgursky@nd.edu}}

	 \author{Stephen E. McKeown}
\address{University of Texas at Dallas \\
	Department of Mathematical Sciences, FO 35 \\
800 W. Campbell Road \\
Richardson, TX 75080-3021}
\email{\href{mgursky@nd.edu}{mgursky@nd.edu}}

\author{Aaron J. Tyrrell}
\address{Department of Mathematics and Statistics \\
Texas Tech University \\
Lubbock, TX 79409}
\email{\href{mgursky@nd.edu}{mgursky@nd.edu}}

\begin{document}

\maketitle

\begin{abstract}  We prove rigidity and gap theorems for self-dual and even Poincar\'e-Einstein metrics in dimension four.  As a corollary, we give an obstruction to the existence of self-dual Poincar\'e-Einstein metrics in terms of conformal invariants of the boundary and the topology of the bulk.  As a by-product of our proof we identify a new scalar conformal invariant of three-dimensional Riemannian manifolds.
\end{abstract}

\section{Introduction}

Let $X^{n+1}$ be the interior of a compact $(n+1)$-dimensional manifold $\overline{X}$ with non-empty boundary $M = \partial X$.  A Riemannian metric $g_{+}$ defined in $X$ is called {\em conformally compact} if there is a defining function $\rho$ for the boundary (i.e., $\rho > 0$ in $X$, $M = \{ \rho = 0 \}$, and $d\rho \neq 0$ on $M$) such that $\bg = \rho^2 \bg$ defines a Riemannian metric (of some degree of regularity) on $\overline{X}$.
Since defining functions are obviously not unique, a conformally compact metric $g_{+}$ determines a conformal class of metrics on the boundary called the {\em conformal infinity} of $(X,g_{+})$.

If $g_{+}$ satisfies the Einstein condition
\begin{align} \label{PED}
Ric(g_{+}) = - n g_{+},
\end{align}
then $g_{+}$ is called a {\em Poincar\'e-Einstein} metric.   The model of Poincar\'e-Einstein metrics is the Poincar\'e model of hyperbolic space on the unit ball $B^{n+1} = \{ x \in \mathbb{R}^{n+1} \, : \, |x| < 1 \}$, where
\begin{align} \label{PBall}
g_H = \frac{4}{(1-|x|^2)^2} ds^2
\end{align}
and $ds^2$ is the Euclidean metric.

In this article we restrict our attention to $n = 3$ and consider two special, but important, classes of Poincar\'e-Einstein manifolds. In each case, we establish gap and rigidity theorems.

We will extensively use the expansions of Poincar\'e-Einstein metrics at the boundary, the so-called Fefferman-Graham expansions. These expansions are guaranteed to exist to infinite order so long as $\bg$ has
at least $C^2$ regularity up to the boundary; see \cite{cdls,andreg}. Although it is known that Einstein metrics exist not satisfying either the hypothesis or the conclusion of this result (\cite{leebah}),
we will assume as a matter of definition that all Poincar\'{e}-Einstein metrics have at least a $C^2$ compactification.

\subsection{Self-Dual Poincar\'e-Einstein metrics}  Suppose $X = X^4$ is four-dimensional and oriented. If $g$ is any Riemannian metric on $X$, then under the action of the Hodge-$\ast$ operator the bundle of self-dual two-forms splits into two rank-three sub-bundles
of {\em self-dual} and {\em anti-self dual} two-forms
\begin{align} \label{lmp}
\Lambda^2(X) = \Lambda^2_{+}(X) \oplus \Lambda^2_{-}(X),
\end{align}
corresponding to the $+1$ and $-1$ eigenspaces of $\ast$.   The Weyl curvature tensor of $g$, viewed as a linear map $W_g : \Lambda^2(X) \rightarrow \Lambda^2(X)$, preserves the splitting (\ref{lmp}).  As a consequence, there are well defined bundle maps $W_g^{\pm} : \Lambda^2_{\pm}(X) \rightarrow \Lambda^2_{\pm}(X)$.  We say that $g$ is {\em sef-dual} if $W^{-} \equiv 0$.

Examples of self-dual Poincar\'e-Einstein metrics (henceforth SDPE metrics) include, of course, the hyperbolic metric on the unit ball.
Pedersen \cite{Pedersen} gave an explicit family of $SU(2)$-invariant self-dual Poincar\'e-Einstein metrics on $B^4$ which we will discuss in more detail below (see also \cite{Matsumoto}).  More generally, given any $SU(2)$-invariant conformal class $[g]$ on $S^3$, Hitchin \cite{Hitchin} proved the existence of a SDPE metric $g_{+}$ in $B^4$ whose conformal infinity is $[g]$.

LeBrun \cite{LeBrunCG} showed that {\em locally}, the existence of self-dual Poincar\'e-Einstein metrics is unobstructed when the boundary metric is real analytic.  More precisely, if $M^3$ is real analytic and $g$ is a real analytic metric on $M^3$, then there is an $\epsilon > 0$ and a SDPE metric $g_{+}$ defined on $X = M^3 \times (0,\epsilon)$ whose conformal infinity is given by $(M^3,[g])$.  Subsequently, Fefferman-Graham (\cite{FG}, Chapter 5) gave a different proof of this fact.

There is, however, a {\em global} obstruction to the existence of SDPE metrics with given conformal boundary, which follows from the Atiyah-Patodi-Singer index theorem:
\begin{align} \label{sig}
\int_X \left( |W^{+}_{g_{+}}|^2 - |W^{-}_{g_{+}}|^2 \right) \, dv_{g_{+}} = 12 \pi^2 \left( \tau(X) - \eta(M,[g]) \right),
\end{align}
where $\tau(X)$ is the signature of $X$ and $\eta(M,[g])$ is the eta-invariant of the boundary. (Here and throughout the paper, the operator norm of the Weyl tensor is used; i.e., $|W^{\pm}|^2 = \frac{1}{4} W^{\pm}_{ijk\ell} \left(W^{\pm}\right)^{ijk\ell}$).  If $g_{+}$ is SDPE, then (\ref{sig}) implies
\begin{align} \label{sigsd}
\int_X |W^{+}_{g_{+}}|^2  \, dv_{g_{+}} = 12 \pi^2 \left( \tau(X) - \eta(M,[g]) \right).
\end{align}
In particular,
\begin{align} \label{sigob}
\tau(X) \geq \eta(M,[g]),
\end{align}
with equality if and only if $(X,g_{+})$ is hyperbolic.  We will refer to this inequality as the {\em signature obstruction}.

One of our main results gives a new obstruction to the existence of self-dual Poincar\'e-Einstein metrics in terms of topological invariants of the bulk and conformal invariants of the boundary (see Theorem \ref{SDThm} below).  One such invariant is the well known Yamabe invariant; the other invariant is ``local" (i.e., can be expressed in terms of the curvature and its derivatives).  To describe this invariant we need to introduce some additional notation.

Let $(M^3, g)$ be a closed, three-dimensional Riemannian manifold.   Let $P$ denote the Schouten tensor of $g$, and
\begin{align*}
C_{ijk} = \nabla_k P_{ij} - \nabla_j P_{ik}
\end{align*}
the Cotton tensor.  We can equivalently view $C$ as a symmetric two-tensor by defining
\begin{align*}
\mathcal{C}_{ij} &= \mu_i^{\ k\ell} C_{j k \ell},
\end{align*}
where $\mu$ is the volume form. It is well known that $C$ is a conformal invariant; hence $\mathcal{C}$ is:
\begin{align*}
\tg = e^{2w}g \ \Rightarrow \ \mathcal{C}_{\tg} = e^{-4w} \mathcal{C}_g.
\end{align*}
It is also well known that in dimension three, vanishing of the Cotton tensor is the obstruction to a metric being locally conformally flat.

One can construct other (non-obvious) conformal invariants from the Cotton tensor, and two such invariants will play a key role in our work on SDPE metrics:

\begin{theorem} \label{CIThm} Let $(M,g)$ be a three-dimensional Riemannian manifold.  Let $C = C_g$ denote the Cotton tensor, and $\mu = \mu_g$ the volume form with respect to $g$.  Also, define the symmetric two-tensors
\begin{align*}
V_{ij} &= \nabla^k C_{ijk}, \\
\mathcal{C}_{ij} &= \mu_i^{\ k\ell} C_{j k \ell},
\end{align*}
where $\mu$ is the volume form of $(M,g)$.  Then $\langle V, \mathcal{C} \rangle_g$ and $|\mathcal{C}|_g^2$ are pointwise conformal invariants:
if $\widetilde{g} = e^{2w_0} g$,
\begin{align} \label{changes} \begin{split}
\langle \widetilde{V}, \widetilde{\mathcal{C}} \rangle_{\widetilde{g}} &= e^{-7w_0} \langle V, \mathcal{C} \rangle_g, \\
|\widetilde{\mathcal{C}}|^2_{\widetilde{g}} &= e^{-6w_0} |\mathcal{C}|^2_g.
\end{split}
\end{align}
Consequently, when it converges
\begin{align} \label{Iz}
I(M^3,[g]) = \int_M \dfrac{ \langle V, \mathcal{C} \rangle }{|\mathcal{C}|^{4/3}} \, dv_g
\end{align}
is an integral conformal invariant.
\end{theorem}

The conformal invariance of $\langle V, \mathcal{C}\rangle$ does not seem to be known to experts, and its construction may be of independent interest.

Since the integral in (\ref{Iz}) may not converge, in general we define
\begin{align} \label{genIdef}
	I_0(M^3, [g]) := \limsup_{\epsilon \to 0} \int_M \dfrac{ \langle V, \mathcal{C} \rangle }{\left( \epsilon + |\mathcal{C}|^2\right)^{2/3}} \, dv_g \in [-\infty,\infty].
\end{align}
For example, when $(M^3,g)$ is locally conformally flat (hence $\mathcal{C} = 0$), $I_0(M^3,[g]) = 0$.

Our next result shows that when the invariant $I_0$ and the Yamabe invariant of the conformal infinity of a SDPE metric are both positive, then the signature obstruction can be sharpened:

\medskip

\begin{theorem} \label{SDThm}  Let $(X^4,g_{+})$ be an oriented, self-dual, four-dimensional Poincar\'e-Einstein manifold.  Let $M^3 = \partial X^4$, and $(M^3,[g])$ denote the conformal infinity of $(X^4,g_{+})$.  Assume  \vskip.1in

\noindent $(i)$  $I_0(M^3,[g]) \geq 0$,  \vskip.1in

\noindent $(ii)$  $Y(M^3, [g]) > 0$.    \vskip.1in

Then either $g_{+}$ is hyperbolic, or
\begin{align} \label{tauSD}
\tau(X^4) \geq \eta \left(M^3, [g]\right) + \frac{1}{3} \chi(X^4).
\end{align}
\end{theorem}

Specializing to the ball, we have

\begin{corollary} \label{SDBallCor}  Let $g_{+}$ be a self-dual, four-dimensional Poincar\'e-Einstein metric on $B^4$ with conformal infinity $(S^3,[g])$.   Assume  \vskip.1in

\noindent $(i)$  $I_0(S^3,[g]) \geq 0$,  \vskip.1in

\noindent $(ii)$  $Y(S^3, [g]) > 0$.    \vskip.1in

Then either $g_{+}$ is hyperbolic, or
\begin{align} \label{tauSD2}
\eta \left(S^3, [g]\right) \leq -\frac{1}{3}.
\end{align}
\end{corollary}

Interestingly, Theorem \ref{SDThm} also gives obstructions to the existence of self-dual Poincar\'e-Einstein fillings:

\begin{corollary} \label{SDCor} Let $(M^3,[g])$ be a conformal three-manifold satisfying  \vskip.1in

\noindent $(i)$  $I_0(M^3,[g]) \geq 0$,  \vskip.1in

\noindent $(ii)$  $Y(M^3, [g]) > 0$.    \vskip.1in

Then there are infinitely many smooth, non-diffeomorphic manifolds $Y^4$ with the following properties:  \vskip.05in

\noindent $\bullet$  $\partial Y^4 = M^3$, \vskip.1in

\noindent $\bullet$  $Y^4$ and $(M^3,[g])$ satisfy the signature obstruction (\ref{sigob}), \vskip.1in

\noindent $\bullet$ $Y^4$ does \underline{not} admit a self-dual Poincar\'e-Einstein metric $g_{+}$ whose conformal infinity is given by $(M^3,[g])$.

\end{corollary}

\medskip

For all $k \geq 1$, the connected sum of $k$ copies of $S^2 \times S^1$ admit locally conformally flat metrics with positive scalar, and therefore satisfy the assumptions of Corollary \ref{SDCor}:

\begin{corollary} \label{SDCor2}  Let $g$ be a locally conformally flat metric of positive scalar curvature on $M^3 = k \left( S^2 \times S^1\right)$.  Then there are infinitely many smooth, non-diffeomorphic manifolds $Y^4$ satisfying the conclusions of Corollary \ref{SDCor}.
\end{corollary}

In Pedersen's construction of $SU(2)$-invariant SDPE metrics on $B^4$, the conformal infinities are given by the Berger spheres, whose construction we briefly recall.  Let $\{ E_1, E_2, E_3 \}$ be the basis of the Lie algebra $\mathfrak{su}(2)$ given by
\begin{align*}
E_1 = \begin{pmatrix} i & 0 \\
0 & -i
\end{pmatrix},
E_2 = \begin{pmatrix} 0 & 1 \\
-1 & 0
\end{pmatrix},
E_3 = \begin{pmatrix} 0 & i \\
i & 0
\end{pmatrix},
\end{align*}
and let $\{ E^1, E^2, E^3 \}$ be the dual basis of one-forms.  The $1$-parameter family of left-invariant metrics
\begin{align*}
g_{\epsilon} = \epsilon E^1 \otimes E^1 + E^2 \otimes E^2 + E^3 \otimes E^3
\end{align*}
on $SU(2) \simeq S^3$ are called the {\em Berger metrics}.  When $\epsilon = 1$, then $g_1 = g_0$, the round metric on $S^3$.  In particular, Pedersen's examples give a $1$-parameter family $(g_{+})_{\epsilon}$ of SDPE metrics on $B^4$ with $(g_{+})_1 = g_H$, the hyperbolic metric.  A tedious calculation gives
\begin{align} \label{IBerger}
I(S^3, [g_{\epsilon}]) = - 12 \sqrt[3]{6}\pi^2 \epsilon^4 |\epsilon - 1|^{2/3} \leq 0,
\end{align}
with equality holding only for $\epsilon = 1$; i.e., $g_1 = g_0$.  In particular, Theorem \ref{SDThm} does not apply to these metrics.

\medskip

\subsection{Even Poincar\'e-Einstein metrics}  Let $(X^4,g_{+})$ be a four-dimensional Poincar\'e-Einstein manifold with conformal infinity $(M^3,[g])$.   By the work of Graham-Lee (see \cite{GrahamLee,RobinRNV}),
given $\hg \in [g]$ there is a unique defining function $r > 0$ in a neighborhood of $M$ such that $g_{+}$ can be expressed as
\begin{align} \label{FG1}
g_{+} = r^{-2} \left( dr^2 + h_r \right),
\end{align}
where $h_r$ is a family of metrics on $M$ with an expansion of the form
\begin{align} \label{FG2}
h_r = \hg + g^{(2)} r^2 + g^{(3)} r^3 + \cdots,
\end{align}
with $g^{(i)}$ tensors on $M$.  The tensor $g^{(3)}$ is formally undetermined, and can be viewed as the Neumann data corresponding to the Dirichlet data $\hg$.

\begin{definition}  \label{EvenDef} We say that $g_{+}$ is {\em even} if $g^{(3)} = 0$.  \end{definition}

It follows from \cite[p. 34 and Lemma 2.2]{RobinRNV} that the property of evenness does not depend on the choice of conformal representative of $[g]$, and hence is conformally invariant.
Also, as the name suggests, evenness implies that $g^{(2k+1)} = 0$ for all $k \geq 0$ in the expansion (\ref{FG2}); see the same source.   Hyperbolic metrics and their quotients are examples of even metrics.

Anderson \cite{An2} showed that even metrics arise in the first variation of the renormalized volume.   More precisely, suppose $h$ is an infinitesimal Einstein deformation (i.e., $h$ is in the kernel of the linearized Einstein operator with respect to $g_{+}$; see Section 2 of \cite{An2} for details).  If $h_0= h\vert_{TM}$ denotes the induced variation of the boundary metric $\hg$, then
\begin{align} \label{Vp}
V'(h) = \frac{d}{dt} V\left(g_{+} + th\right)\big|_{t=0} = -\frac{1}{4} \oint_M \langle h_0 , g^{(3)} \rangle_{\hg} \, dA_{\hg}
\end{align}
(see Theorem 2.2 of \cite{An2}).   In particular, we see that if $g_{+}$ is even, then it is a critical point of the renormalized volume functional.

The formula (\ref{Vp}) follows from Anderson's formula for the renormalized volume:
\begin{align} \label{AV}
V = \frac{4}{3} \pi^2 \chi(X) - \frac{1}{6} \int_X |W_{g_{+}}|^2 \, dv_{g_{+}}.
\end{align}
Notice this gives the topological bound
\begin{align} \label{Vupper}
V \leq \frac{4}{3} \pi^2 \chi(X),
\end{align}
with equality if and only if $g_{+}$ is hyperbolic.  In particular, for any Poincar\'e-Einstein metric on the ball $B^4$ we have
\begin{align*}
V \leq \frac{4}{3} \pi^2,
\end{align*}
and equality it only attained for the hyperbolic metric.

Our next result gives a ``gap'' result for the renormalized volume for even metrics:

\begin{theorem} \label{EvenThm}  Let $(X^4, g_{+})$ be an even Poincar\'e-Einstein manifold with conformal infinity $(M^3,[g])$.  If $Y(M^3,[g]) > 0$, then either $g_{+}$ is hyperbolic, or
\begin{align} \label{g3gap}
V \leq \min\left\{ \frac{2}{3} \pi^2 \chi(X^4),  - \frac{\sqrt{6}}{2} Y(M^3,[g])^{3/2} + \frac{4}{3} \pi^2 \chi(X^4) \right\}.
\end{align}
In particular, if $g_{+}$ is an even Poincar\'e-Einstein metric on $B^4$ then
\begin{align} \label{ballgap}
V \leq \min \left\{ \frac{1}{2} V_0, V_0 - \frac{\sqrt{6}}{2} Y(M^3,[g])^{3/2} \right\},
\end{align}
where $V_0 = \frac{4}{3}\pi^2$ is the renormalized volume of the hyperbolic metric.
\end{theorem}

By a result of Chang-Qing-Yang \cite{CQY}, if $X^4$ is not diffeomorphic to $B^4$ and $M^3$ is not diffeomorphic to $S^3$, then
\begin{align*}
V \leq \frac{2}{3} \pi^2 \chi(X^4).
\end{align*}
Therefore, (\ref{g3gap}) is an improvement on their inequality only when $X = B^4$ and $M = S^3$, or if the Yamabe invariant of the conformal infinity is sufficiently large.  The dependence of the upper bound on the Yamabe invariant of the conformal infinity relies on a result of Chang-Ge \cite{CG}, which gives a comparison between the Yamabe invariant of the conformal infinity and the ``type-I'' Yamabe invariant of the interior (see Section \ref{MYSec}).

An interesting example is given by $X^4 = B^4 / \mathbb{Z}$ and $M^3 = S^2 \times S^1$.  In this case $X^4$ admits a hyperbolic metric obtained as the quotient of the standard hyperbolic metric on $B^4$ along an isometry.  These metrics are obviously even, and the renormalized volume $V = 0$.  For any other (non-hyperbolic) even metric on $B^4 / \mathbb{Z}$, (\ref{g3gap2}) implies
\begin{align} \label{g3gap2}
V \leq - \frac{\sqrt{6}}{2} Y(M^3,[g])^{3/2} < 0.
\end{align}

There is an obvious parallel between the question of the existence of non-hyperbolic even Poincar\'e-Einstein metrics on $B^4$, and the existence of non-round Einstein metrics on $S^4$.   In both cases, the metrics are critical points of a natural geometric variational problem.  Furthermore, the associated canonical metrics (the hyperbolic metric in the former case, and the round metric in the latter) are isolated; i.e., in a sufficiently small neighborhood there are no other examples.  For both cases there is also a notion of ``positivity":  for even metrics the sign of the Yamabe invariant of the conformal infinity, and for Einstein metrics the sign of the Einstein constant.  We end with the following natural question:

\begin{question} Suppose $g_{+}$ is an even Poincar\'e-Einstein metric on $B^4$ whose conformal infinity is of positive Yamabe type.  Is $g_{+}$ hyperbolic?
\end{question}

\medskip

\subsection{Organization of the paper}  In Section \ref{MYSec} we prove an elementary inequality for the Yamabe invariant that follows from the Chern-Gauss-Bonnet formula.  We also record an estimate for the Yamabe invariant of Poincar\'e-Einstein manifolds due to Chang and Ge which will be used in the proof of Theorem \ref{EvenThm}.   In Section \ref{WFSec} we prove a key Weitzenb\"ock formula for the (weighted) Weyl tensor.

In Sections \ref{SDProof} and \ref{EvenProof} we present the proofs of Theorems \ref{SDThm} and \ref{EvenThm}, assuming certain expansions for the Weyl tensor of a Poincar\'e-Einstein metric.  These expansions, which build on the formulas in Chapter 5 of \cite{FG}, are worked out in Sections \ref{Expansions1} and \ref{Expansions2}.

Finally, in Section \ref{CISec} we give a proof of Theorem \ref{CIThm}.
\medskip

\subsection{Acknowledgements}  The authors would like to thank Alice Chang for informing us of her work with Yuxin Ge, which we quote in Lemma \ref{CGLemma}, and for numerous enlightening conversations.  The first author acknowledges the support of NSF grant DMS-2105460.  The second author acknowledges the support of Simons Foundation grant 966614.

\bigskip

\section{Estimates of the Yamabe invariant}  \label{MYSec}

In this section prove two estimates for the Yamambe invariant in our setting that will be needed in the proofs of the main results.

Let $(X,M,h)$ be a compact Riemannian four-manifold with non-empty boundary $M = \partial X$.  Let $R_h$ denote the scalar curvature of $h$ and $H = H_h$ the mean curvature of $M$ with respect to $h$.  Define $\mathcal{E}_h : W^{1,2}(X) \rightarrow \mathbb{R}$ by
\begin{align} \label{E}
\mathcal{E}_h[\phi] = \int_X \left( |\nabla_h \phi|^2 + \frac{1}{6} R_h \phi^2 \right) \, dv_h + \frac{1}{3} \oint_{M} H \phi^2 \, dA_h.
\end{align}
The quantity $\mathcal{E}$ is conformally invariant in the sense that if $\wh = e^{2w} h$, then
\begin{align} \label{ctl}
\mathcal{E}_{\wh}[\phi] = \mathcal{E}_h[e^w \phi].
\end{align}
We also define the Yamabe functional
\begin{align} \label{YE}
\mathcal{Y}_h[\phi] = \dfrac{\mathcal{E}_h[\phi]}{\left(\int_X |\phi|^4 \, dV_h \right)^{1/2}}.
\end{align}
The (type I-) Yamabe invariant is defined by
\begin{align} \label{Y1}
Y_1(X,M,[h]) = \inf_{\phi \in W^{1,2}(X) \setminus \{0\}} \mathcal{Y}_h[\phi].
\end{align}
By standard conformal transformation laws, an equivalent definition is given by
\begin{align} \label{Y1EH}
Y_1(X,M,[h]) = \inf_{\tilde{h} \in [h]} \dfrac{\frac{1}{6} \int_X   R_{\tilde{h}} \, dv_{\tilde{h}} + \frac{1}{3} \oint_{M} H_{\tilde{h}} \, dA_{\tilde{h}}}{\left( vol(X,\tilde{h}) \right)^{1/2}}.
\end{align}
A function $v \in W^{1,2}$ attaining $Y_1(X,M,[h])$ defines a (smooth) conformal metric $h_1 = v^2 h$ satisfying
\begin{align} \label{ELY1} \begin{split}
R_{h_1} &= \frac{1}{6} Y_1(X,M,[h]) \ \ \mbox{in }X, \\
H_{h_1} &= 0 \ \ \mbox{on }M
\end{split}
\end{align}
(see \cite{EscobarJDG}).

Using the Chern-Gauss-Bonnet formula, we obtain the following elementary inequality for the Yamabe invariant:

\begin{lemma} \label{YLemma}   Let $(X,M,h)$ be a compact Riemannian four-manifold with non-empty boundary $M = \partial X$.  Assume \vskip.1in

\noindent $(i)$  $M$ is umbilic with respect to $h$.   \vskip.1in

\noindent $(ii)$ $Y_1(X,M,[h]) \geq 0$. \vskip.1in

Then
\begin{align} \label{CY1}
8 \pi^2 \chi(X) \leq  \int_X |W_{h}|^2 \, dv_{h} +  \frac{3}{2} Y_1(X,M,[h])^2.
\end{align}

\end{lemma}

\begin{proof}  By conformal invariance we may assume that $M$ is minimal with respect to $h$ (see Lemma 1.1 of \cite{EscobarJDG}).   As in the proof of Theorem 6.1 of \cite{EscobarJDG}, let $\{ v_i \}$ be a minimizing sequence for $Y_1(X,M,[h])$ obtained by the standard sub-critical regularization procedure.  More precisely, each $v_i$ is a solution of the PDE
\begin{align} \label{sub} \begin{split}
- \Delta_h v_i + \frac{1}{6} R_h v_i &= Y_1(X,M,[h]) v_i^{3 - \delta_i} \ \ \mbox{in }X, \\
\frac{\partial v_i}{\partial \nu} &= 0 \ \ \mbox{on }M,
\end{split}
\end{align}
where $\nu$ is the outward unit normal with respect to $h$ and $\delta_i \to 0$ as $i \to \infty$.  Moreover, we may assume that $v_i$ is normalized so that $\int_X v_i^4 \, dv_h = 1$. Here and throughout,
$\Delta_h = \div_h\grad_h$.

By standard regularity properties (see \cite{Cherrier}, Theorem 1), we may assume $v_i \in C^{\infty}(\overline{X})$ and $v_i > 0$.  By (\ref{sub}), the scalar curvature and mean curvature of $h_i = v_i^2 h$ are given by
\begin{align} \label{RHi} \begin{split}
R_{h_i} &= 6 Y_1(X,M,[h]) v_i^{- \delta_i} \ \ \mbox{in X}, \\
H_{h_i} &= 0 \ \ \mbox{on }M.
\end{split}
\end{align}
In particular, since the boundary is minimal with respect to $h_i$, it is totally geodesic.  Therefore, by the Chern-Gauss-Bonnet formula
\begin{align} \label{fip3} \begin{split}
8 \pi^2 \chi(X) &= \int_X |W_{h_i}|^2 \, dv_{h_i} - \frac{1}{2} \int_X |E_{h_i}|^2 \, dv_{h_i} + \frac{1}{24} \int_X R_{h_i}^2 \, dv_{h_i} \\
&\leq \int_X |W_{h_i}|^2 \, dv_{h_i} + \frac{1}{24} \int_X R_{h_i}^2 \, dv_{h_i}.
\end{split}
\end{align}
By (\ref{RHi}) and the volume normalization,
\begin{align} \begin{split}
\int_X R_{h_i}^2 \, dv_{h_i} &= 36 Y_1(X,M,[h])^2 \int_X v_i^{4 - 2\delta_i} \, dv_h \\
&\leq 36 Y_1(X,M,[h])^2 \left( \int_X v_i^4 \, dv_h \right)^{1 - \delta_i/2} \left( \int_X dv_h \right)^{\delta_i/2} \\
&=  36 Y_1(X,M,[h])^2  \vol\left(h\right)^{\delta_i/2}.
\end{split}
\end{align}
Combining this with (\ref{fip3}) and using the conformal invariance of the $L^2$-norm of the Weyl tensor, we find
\begin{align} \label{fip4}
8 \pi^2 \chi(X) \leq  \int_X |W_{h}|^2 \, dv_{h} +  \frac{3}{2} Y_1(X,M,[h])^2  \vol\left(h\right)^{\delta_i/2}.
\end{align}
Letting $i \to \infty$, we arrive at (\ref{CY1}).
\end{proof}

Now let $(X,g_{+})$ be a four-dimensional Poincar\'e-Einstein manifold with conformal infinity $(M^3,[g])$.  Various estimates for the Yamabe invariant of the boundary relative to the Yamabe invariant of a compactification appear in \cite{GH}, \cite{CLW}, \cite{F}, \cite{WW}, \cite{CG}.  The following estimate of Chang-Ge will be used in the proof of Theorem \ref{EvenThm}:

\begin{lemma} \label{CGLemma} (See \cite{CG}) Let $(X,g_{+})$ be a four-dimensional Poincar\'e-Einstein manifold with conformal infinity $(M^3,[g])$.  Assume $Y(M^3,[g]) \geq 0$.  Then
\begin{align} \label{ChangGe}
Y_1(X,M,[\bg])^2 \geq 2 \sqrt{6} Y(M^3,[g])^{3/2},
\end{align}
where $\bg = r^2 g_{+}$ is any (smooth) compactification of $(X,g_{+})$.
\end{lemma}

\medskip

\section{A Weitzenb\"ock formula} \label{WFSec}

The main result of this section is the following:

\begin{theorem} \label{WFThm} Let $(X^4,g_{+})$ be an oriented four-dimensional Poincar\'e-Einstein manifold.  Let $\rho > 0$ be a defining function, and $\bar{g} = \rho^2 g_{+}$.
Define
\begin{align*}
	Z_{\bar{g}}^{+} &= \rho W_{g_{+}}^{+}.
\end{align*}
Then $Z^{+} = Z^{+}_{\bar{g}}$ satisfies
\begin{align} \label{WFZp}
\frac{1}{2} \Delta |Z^{+}|^2 = |\nabla Z^{+}|^2 - 6 \tr \left(  W^{+}  \circ \left( Z^{+} \right)^2 \right) + \frac{1}{2} R |Z^{+}|^2,
\end{align}
where the covariant derivatives and curvature are with respect to the metric $\bar{g}$.  Also, away from the zero locus of $|Z^{+}|$,
\begin{align} \label{Kato}
|\nabla Z^{+}|^2 \geq \frac{5}{3} |\nabla |Z^{+}||^2.
\end{align}
\end{theorem}

Theorem \ref{WFThm} will follow from a more general result.   Let $(X,g_0)$ be an oriented, four-dimensional manifold such that the self-dual Weyl tensor is harmonic:
\begin{align} \label{sdh}
\left( \delta_{g_0} W^{+}_{g_0} \right)_{j k \ell} = \nabla_{g_0}^m \left( W^{+}_{g_0} \right)_{mjk\ell} = 0.
\end{align}
Given a conformal metric $g = e^{2w}g_0$, the conformal transformation formula for the Riemannian connection implies
\begin{align} \label{harch}
	\left( \delta_g W_g^{\pm} \right)_{jk\ell} = e^{-2w} \left\{ \left( \delta_{g_0} W^{\pm}_{g_0} \right)_{j k \ell}  - \nabla_{g_0}^m w \left( W^{\pm}_{g_0}\right)_{mjk\ell} \right\}.
\end{align}
In particular, the condition (\ref{sdh}) is not conformally invariant.  However, the formula (\ref{harch}) easily implies the following:

\medskip

\begin{lemma}  \label{ZdefLemma} Suppose $(X,g_0)$ is an oriented, four-dimensional manifold whose self-dual Weyl tensor is harmonic.  Given a conformal metric $g = e^{2w}g_0$, let
\begin{align} \label{Zpm}
Z^{+}_g = e^w W_{g_0}^{+} \in \Gamma\left( \mathcal{W}^{+}\right),
\end{align}
where $\mathcal{W}^{+} \subset \mathcal{W}$ is the sub-bundle of self-dual algebraic Weyl tensors.  Then $Z^{+} = Z^{+}_g$ is harmonic with respect to $g$:
\begin{align} \label{Zhar}
\delta_g Z^{+} = 0.
\end{align}
\end{lemma}

\medskip

In \cite{Derdzinski}, Derdzinski showed that the condition $\delta_{g_0} W^{+}_{g_0} = 0$ implies that $W^{+}_{g_0}$ satisfies the following Weitzenb\"ock formula:
\begin{align} \label{Derd} \begin{split}
\frac{1}{2} \Delta_{g_0} |W^{+}_{g_0}|^2 &= |\nabla_{g_0} W^{+}_{g_0}|^2 - 18 \det W^{+}_{g_0} + \frac{1}{2} R_0 |W^{+}_{g_0}|^2 \\
&= |\nabla_{g_0} W^{+}_{g_0}|^2 -  6 \tr_{g_0} \left( W^{+}_{g_0} \circ \left( W^{+}_{g_0} \right)^2 \right) + \frac{1}{2} R_0 |W^{+}_{g_0}|^2, 
\end{split}
\end{align}
where the norms are with respect to $g_0$.  If $g = e^{2w}g_0$, since $Z^{+}$ is harmonic with respect to $g$, it satisfies a similar Weitzenb\"ock formula:

\begin{corollary}  \label{WFCor} Suppose $(X,g_0)$ is an oriented, four-dimensional manifold whose self-dual Weyl tensor is harmonic.  Given a conformal metric $g = e^{2w}g_0$,
let $Z^{+} = Z^{+}_g$ be defined as in (\ref{Zpm}).  Then $Z^{+}$ satisfies
\begin{align} \label{WFZp2}
\frac{1}{2} \Delta_g |Z^{+}|^2 = |\nabla Z^{+}|^2 - 6 \tr \left( W^{+} \circ \left( Z^{+} \right)^2 \right) + \frac{1}{2} R |Z^{+}|^2.
\end{align}
\end{corollary}

\begin{proof} It is possible to give a proof by the same methods of the proof (\ref{Derd}) in \cite{Derdzinski}.  However, it is much more straightforward to simply use the fact that $Z^{+} = e^w W^{+}_{g_0}$, and use the standard conformal transformation formulas for the Riemannian connection and scalar curvature to observe
\begin{align*}
|\nabla Z^{+}|_g^2 &= e^{-8w} \Big\{ |\nabla_{g_0} W^{+}_{g_0}|_{g_0}^2 - 5 \langle \nabla_{g_0} w, \nabla_{g_0} |W^{+}_{g_0} \rangle_{g_0} + 15 |\nabla_{g_0} w|_{g_0}^2 |W^{+}_{g_0}|^2_{g_0} \Big\}, \\
\Delta_g |Z^{+}|_g^2 &= e^{-8w} \Big\{ \Delta_{g_0} |W^{+}_{g_0}|^2_{g_0} - 6 ( \Delta_{g_0} w ) |W^{+}_{g_0}|^2_{g_0} - 10 \langle \nabla_{g_0} w, \nabla_{g_0} |W^{+}_{g_0} \rangle_{g_0} + 24 |\nabla_{g_0} w|_{g_0}^2 |W^{+}_{g_0}|^2_{g_0} \Big\}, \\
R_g |Z^{+}|_g^2 &= e^{-8w} \Big\{ R_{g_0} |W^{+}_{g_0}|^2_{g_0} - 6( \Delta_{g_0} w ) |W^{+}_{g_0}|^2_{g_0} - 6 |\nabla_{g_0} w|_{g_0}^2 |W^{+}_{g_0}|^2_{g_0} \Big\}.
\end{align*}
Combining these with (\ref{Derd}) we obtain (\ref{WFZp2}). 
\end{proof}

\medskip

\begin{corollary} \label{KatoCor} Suppose $(X,g_0)$ is an oriented, four-dimensional manifold whose self-dual Weyl tensor is harmonic.  Given a conformal metric $g = e^{2w}g_0$,
let $Z^{+} = Z^{+}_g$ be defined as in (\ref{Zpm}).  Then away from the zero locus of $|Z^{+}|$,
\begin{align} \label{Kato2}
|\nabla Z^{+}|^2 \geq \frac{5}{3} |\nabla |Z^{+}||^2.
\end{align}
\end{corollary}

\begin{proof}  The Kato inequality (\ref{Kato}) was proved for the self-dual Weyl tensor in \cite{GL} (see also \cite{CGH}), but the proof for $Z^{+}$ is obviously the same.
\end{proof}

\medskip

\begin{proof}[The proof of Theorem \ref{WFThm}]  Suppose $(X^4, g_{+})$ is an oriented, four-dimensional Poincar\'e-Einstein manifold.  The Einstein condition implies that the self-dual Weyl tensor $W^{+}_{g_{+}}$ is harmonic.  Therefore, given a defining function $\rho > 0$, we can apply Corollaries \ref{WFCor} and \ref{KatoCor} with $g_0 = g_{+}$ and $e^w = \rho$, and Theorem \ref{WFThm} follows.
\end{proof}

\bigskip

\section{The proof of Theorem \ref{SDThm} and its corollaries} \label{SDProof}

\begin{proof}[The proof of Theorem \ref{SDThm}]  Suppose $(X^4,g_{+})$ is an oriented, self-dual, four-dimensional Poincar\'e-Einstein manifold with conformal infinity $(M^3,[g])$.  We assume   \vskip.1in

\noindent $(i)$  $I_0(M^3,[g]) \geq 0$,  \vskip.1in

\noindent $(ii)$  $Y(M^3, [g]) > 0$.    \vskip.1in

Choose a representative in the conformal infinity (which we again denote $g$) and let $r > 0$ be the special defining function associated to $g$.  Let $\bg = r^2 g_{+}$.   The proof of Theorem \ref{SDThm} is based on the
following proposition:

\begin{proposition}  \label{signProp2}  Under the assumptions of Theorem \ref{SDThm},
\begin{align}  \label{2ndYamabeEst}
Y_1(X,M,[\bg])^2 \leq  \frac{2}{3}  \int_X |W^{+}_{g_{+}}|^2 \, dv_{g_{+}}.
\end{align}
\end{proposition}

Before giving a proof of this proposition, let us show how Theorem \ref{SDThm} follows.  Again, as observed by J. Qing in \cite{Qing}, since $Y(M^3, [g]) > 0$ it follows that $Y_1(X,M,[\bg]) > 0$.  Also, the Fefferman-Graham expansion implies that $M$ is totally umbilic with respect to $\bg$.   Therefore, by Lemma \ref{YLemma} we have
\begin{align} \label{WZ100}
8 \pi^2 \chi(X) \leq  \int_X |W_{\bg}|^2 \, dv_{\bg} + \frac{3}{2} Y_1(X,M,[\bg])^2 = \int_X |W_{g_{+}}|^2 \, dv_{g_{+}}  + \frac{3}{2} Y_1(X,M,[\bg])^2.
\end{align}
Combining with (\ref{2ndYamabeEst}), we obtain
\begin{align} \label{WZ}
8 \pi^2 \chi(X) \leq  2 \int_X |W^{+}_{g_{+}}|^2 \, dv_{g_{+}}.
\end{align}
By the Atiyah-Patodi-Singer index formula,
\begin{align*}
\int_X |W^{+}_{\bg}|^2 \, dv_{\bg} = 12 \pi^2 \left( \tau(X) - \eta(M,[g]) \right),
\end{align*}
hence
\begin{align*}
8 \pi^2 \chi(X) \leq 24 \pi^2 \left( \tau(X) - \eta(M,[g]) \right),
\end{align*}
which gives (\ref{tauSD}).    \end{proof}

\begin{proof}[The proof Proposition \ref{signProp2}]  Let $Z^{+} = r W^{+}_{g_{+}}$.  By Theorem \ref{WFThm}, $Z^{\pm}$ satisfy
\begin{align} \label{WZ3}
\frac{1}{2} \Delta_{\bg} |Z^{\pm}|^2 = |\nabla_{\bg} Z^{\pm}|^2 - 6 \tr \left(  W^{\pm}_{\bg}  \circ \left( Z^{\pm} \right)^2 \right) + \frac{1}{2} R_{\bg} |Z^{\pm}|^2.
\end{align}
Since $Z^{\pm}$ and $W^{\pm}$ are trace-free,
\begin{align*}
\Big| \tr \left(  W^{\pm}_{\bg}  \circ \left( Z^{\pm} \right)^2 \right)\Big| \leq \frac{1}{\sqrt{6}} |W^{\pm}_{\bg}||Z^{\pm}|^2.
\end{align*}
Substituting this into (\ref{WZ3}), we have
\begin{align} \label{WZ4}
\frac{1}{2} \Delta_{\bg} |Z^{\pm}|^2 \geq |\nabla_{\bg} Z^{\pm}|^2  -  \sqrt{6} |W^{\pm}_{\bg}| |Z^{\pm}|^2 + \frac{1}{2} R_{\bg} |Z^{+}|^2.
\end{align}
For $\epsilon > 0$ let
\begin{align} \label{fe}
f_{\epsilon} = \left( \epsilon + |Z^{+}|^2 \right)^{1/6}.
\end{align}
Multiply (\ref{WZ4}) by $f_{\epsilon}^{-4}$, and integrate over $X$:
\begin{align} \label{WFbg3}
\frac{1}{2} \int_X  f_{\epsilon}^{-4} \Delta_{\bg} |Z^{+}|^2  \, dv_{\bg} \geq \int_X f_{\epsilon}^{-4} |\nabla_{\bg} Z^{+}|^2 \, dv_{\bg}  + \frac{1}{2} \int_X \left( R_{\bg} - 2 \sqrt{6} |W^{+}_{\bg}| \right) f_{\epsilon}^{-4} |Z^{+}|^2 \, dv_{\bg}.
\end{align}
For the term on the left, we integrate by parts:
\begin{align} \label{WFbg4} \begin{split}
\frac{1}{2} \int_X  f_{\epsilon}^{-4} \Delta_{\bg} |Z^{+}|^2  \, dv_{\bg} &= - \frac{1}{2} \int_X \langle \nabla_{\bg} \left( f_{\epsilon}^{-4} \right),  \nabla_{\bg} |Z^{+}|^2 \rangle_{\bg} \, dv_{\bg} + \frac{1}{2} \oint_M f_{\epsilon}^{-4} \frac{\partial }{\partial \nu} |Z^{+}|^2 \, dA_{\bg} \\
&= - \frac{1}{2} \int_X \langle \nabla_{\bg} \left( f_{\epsilon}^{-4} \right),  \nabla_{\bg} \left( f_{\epsilon}^6 \right) \rangle_{\bg} \, dv_{\bg} + \frac{1}{2} \oint_M f_{\epsilon}^{-4} \frac{\partial }{\partial \nu} |Z^{+}|^2 \, dA_{\bg} \\
&= 12 \int_X | \nabla_{\bg} f_{\epsilon} |^2 \, dv_{\bg} + \frac{1}{2} \oint_M f_{\epsilon}^{-4} \frac{\partial }{\partial \nu} |Z^{+}|^2 \, dA_g,
\end{split}
\end{align}
where $\nu$ is the outward normal to $M$ with respect to $g = \bg|_M$.  For the first term on the right-hand side of (\ref{WFbg3}), we use the refined Kato inequality $|\nabla_{\bg} Z^{+}|^2 \geq \frac{5}{3}|\nabla_{\bg} |Z^{+}||^2$.
Let $X_Z = \left\{ p \in X: Z^+(p) \neq 0 \right\}$. Then
\begin{align} \label{WFbg5}  \begin{split}
		\int_X f_{\epsilon}^{-4} |\nabla_{\bg} Z^{+}|^2 \, dv_{\bg} &= \int_{X_Z}f_{\epsilon}^{-4}|\nabla_{\bar{g}}Z^+|^2\,dv_{\bg} + \int_{X \setminus X_Z}f_{\epsilon}^{-4}|\nabla_{\bg}Z^+|^2\,dv_{\bg}\\
&\geq \frac{5}{3} \int_X f_{\epsilon}^{-4} |\nabla_{\bg} |Z^{+}||^2 \, dv_{\bg}  \\
&= \left( \frac{5}{3} \right) \left( \frac{1}{4} \right) \int_{X_Z} f_{\epsilon}^{-4}\frac{1}{|Z^{+}|^2} |\nabla_{\bg} \left( |Z^{+}|^2 \right)|^2 \, dv_{\bg}  \\
&= \frac{5}{12}   \int_{X_Z} f_{\epsilon}^{-4}\frac{1}{|Z^{+}|^2} |\nabla_{\bg} \left( f_{\epsilon}^6 \right)|^2 \, dv_{\bg}  \\
&=  15  \int_{X_Z} f_{\epsilon}^{6}\frac{1}{|Z^{+}|^2} |\nabla_{\bg} f_{\epsilon}|^2 \, dv_{\bg}  \\
&\geq  15  \int_{X_Z} f_{\epsilon}^{6}\frac{1}{\left( \epsilon + |Z^{+}|^2\right)} |\nabla_{\bg} f_{\epsilon}|^2 \, dv_{\bg}  \\
&=  15  \int_{X_Z} |\nabla_{\bg} f_{\epsilon}|^2 \, dv_{\bg}\\
&= 15 \int_{X}|\nabla_{\bg}f_{\epsilon}|^2\,dv_{\bg}.
\end{split}
\end{align}
Combining (\ref{WFbg3}), (\ref{WFbg4}), and (\ref{WFbg5}), we obtain
\begin{align} \label{WFbg7}
\frac{1}{2} \oint_M f_{\epsilon}^{-4} \frac{\partial }{\partial \nu} |Z^{+}|^2 \, dA_{\bg} \geq 3  \int_X |\nabla_{\bg} f_{\epsilon}|^2 \, dv_{\bg} + \frac{1}{2} \int_X \left( R_{\bg} - 2 \sqrt{6} |W^{+}_{\bg}| \right) f_{\epsilon}^{-4} |Z^{+}|^2 \, dv_{\bg}.
\end{align}
We can rewrite the last term above as
\begin{align*}
\frac{1}{2} \int_X  &\left( R_{\bg} - 2 \sqrt{6} |W^{+}_{\bg}| \right) f_{\epsilon}^{-4} |Z^{+}|^2 \, dv_{\bg} \\
& = \frac{1}{2} \int_X  \left( R_{\bg} - 2 \sqrt{6} |W^{+}_{\bg}| \right) f_{\epsilon}^{-4} \left( f_{\epsilon}^6 - \epsilon \right) \, dv_{\bg} \\
&= \frac{1}{2} \int_X  \left( R_{\bg} - 2 \sqrt{6} |W^{+}_{\bg}| \right) f_{\epsilon}^2 \, dv_{\bg} - \frac{1}{2} \epsilon \int_X  \left( R_{\bg} - 2 \sqrt{6} |W^{+}_{\bg}| \right) f_{\epsilon}^{-2} \, dv_{\bg}.
\end{align*}
Since $f_{\epsilon} \geq \epsilon^{1/6}$,
\begin{align*}
- \frac{1}{2} \epsilon \int_X \left( R_{\bg} - 2 \sqrt{6} |W^{+}_{\bg}| \right)  f_{\epsilon}^{-2} \, dv_{\bg} = O(\epsilon^{2/3}).
\end{align*}
Consequently, (\ref{WFbg7}) implies
\begin{align} \label{WFbg8}
3  \int_X |\nabla_{\bg} f_{\epsilon}|^2 \, dv_{\bg} + \frac{1}{2} \int_X R_{\bg}   f_{\epsilon}^2 \, dv_g \leq \sqrt{6} \int_X  |W^{+}_{\bg}|   f_{\epsilon}^2 \, dv_{\bg} + \frac{1}{2} \oint_M f_{\epsilon}^{-4} \frac{\partial }{\partial \nu} |Z^{+}|^2 \, dA_{\bg} + O(\epsilon^{2/3}).
\end{align}

It follows from the Fefferman-Graham expansions that the mean curvature of $M$ with respect to $g = \bg|_M$ vanishes.  Therefore, dividing (\ref{WFbg7}) by $3$ we get
\begin{align} \label{WFbg9} \begin{split}
\mathcal{E}_{\bg}[ f_{\epsilon} ] &\leq \frac{\sqrt{6}}{3} \int_X  |W^{+}_{\bg}|   f_{\epsilon}^2 \, dv_{\bg} + \frac{1}{6} \oint_M f_{\epsilon}^{-4} \frac{\partial }{\partial \nu} |Z^{+}|^2 \, dA_g  + O(\epsilon^{2/3}) \\
&\leq \frac{\sqrt{6}}{3} \left(  \int_X  |W^{+}_{g_{+}}|^2 \, dv_{g_{+}} \right)^{1/2} \left( \int_X   f_{\epsilon}^4 \, dv_{\bg} \right)^{1/2} + \frac{1}{6} \oint_M f_{\epsilon}^{-4} \frac{\partial }{\partial \nu} |Z^{+}|^2 \, dA_g  + O(\epsilon^{2/3}).
\end{split}
\end{align}

\begin{claim}  \label{EClaim}
\begin{align} \label{bdyf}
\liminf_{\epsilon \to 0}  \oint_M f_{\epsilon}^{-4} \frac{\partial }{\partial \nu} |Z^{+}|^2 \, dA_g = - 4  I_0(M^3,[g]).
\end{align}
\end{claim}

Assuming the claim for now, by (\ref{WFbg9}) and the assumption that $I_0(M^3,[g]) \geq 0$, we have
\begin{align} \label{WFbg81}
\mathcal{E}_{\bg}[ f_{\epsilon} ]  \leq \frac{\sqrt{6}}{3} \left(  \int_X  |W^{+}_{g_{+}}|^2 \, dv_{g_{+}} \right)^{1/2} \left( \int_X   f_{\epsilon}^4 \, dv_{\bg} \right)^{1/2} + O(\epsilon^{2/3}).
\end{align}
If $g_{+}$ is not hyperbolic, then there is a $c_0 > 0$ (independent of $\epsilon$) such that
\begin{align*}
\int_X   f_{\epsilon}^4 \, dv_{\bg} \geq c_0 > 0.
\end{align*}
Therefore,
\begin{align} \label{WFbg82}
\mathcal{Y}_{\bg}[ f_{\epsilon} ]  \leq \frac{\sqrt{6}}{3} \left(  \int_X  |W^{+}_{g_{+}}|^2 \, dv_{g_{+}} \right)^{1/2}  + O(\epsilon^{2/3}),
\end{align}
and (\ref{2ndYamabeEst}) follows.
\end{proof}

\begin{proof}[The proof of Claim \ref{EClaim}] By the definition of $Z^{+}$ and the conformal transformation law for the Weyl tensor,
\begin{align*}
|Z^{+}|^2 &= |Z^{+}|_{\bg}^2 \\
 &= \big| r W^{+}_{g_{+}} \big|_{\bg}^2 \\
&= r^2 \big| W^{+}_{g_{+}} \big|_{\bg}^2 \\
&= r^2 \big| r^{-2} W^{+}_{\bg} \big|_{\bg}^2 \\
&= r^{-2} | W^{+}_{\bg} |^2_{\bg}.
\end{align*}
By Proposition \ref{WnormProp},
\begin{align*}
|Z^{+}|^2 &= r^{-2} \big\{  r^2 |\mathcal{C} |^2 + 4 r^3 \langle V, \mathcal{C} \rangle + O(r^4) \big\} \\
&= |\mathcal{C} |^2 + 4 r \langle V, \mathcal{C} \rangle + O(r^2),
\end{align*}
where $V$ and $\mathcal{C}$ are with respect to $g = \bg|_{TM}$.  It follows that
\begin{align} \label{feM}
f_{\epsilon} \big|_M = \left( \epsilon + |\mathcal{C} |^2 \right)^{1/6}.
\end{align}
Also, since $r$ is a special defining function, $\nu = - \frac{\partial}{\partial r}\vert_M$, hence
\begin{align} \label{dZM} \begin{split}
\frac{\partial }{\partial \nu} |Z^{+}|^2 &= - \frac{\partial}{\partial r}  \big\{ |\mathcal{C}|^2 + 4 r \langle V, \mathcal{C} \rangle + O(r^2) \big\} \big|_{r=0}\\
&= - 4 \langle V , \mathcal{C} \rangle.
\end{split}
\end{align}
Therefore, combining (\ref{feM}) and (\ref{dZM}), we find
\begin{align*}
\liminf_{\epsilon \to 0}  \oint_M f_{\epsilon}^{-4} \frac{\partial }{\partial \nu} |Z^{+}|^2 \, dA_g &= \liminf_{\epsilon \to 0} \oint_M \dfrac{- 4 \langle V, \mathcal{C} \rangle}{\left( \epsilon + |\mathcal{C} |^2 \right)^{2/3}} \, dA_g \\
&= - 4 \limsup_{\epsilon \to 0} \oint_M \dfrac{\langle V, \mathcal{C} \rangle}{\left( \epsilon + |\mathcal{C} |^2 \right)^{2/3}} \, dA_g \\
&= - 4 I_0(M^3,[g]).
\end{align*}

\end{proof}

\medskip

\begin{proof}[The proof of Corollary \ref{SDCor}]  Suppose $(M^3,[g])$ is a conformal three-manifold satisfying  \vskip.1in

\noindent $(i)$  $I_0(M^3,[g]) \geq 0$,  \vskip.1in

\noindent $(ii)$  $Y(M^3, [g]) > 0$.    \vskip.1in

Let $X_0$ be any smooth four-manifold such that $\partial X_0 = M$.  For $k \geq 0$, let $X_k$ be the manifold with boundary obtained by deleting a ball in the interior of $X_0$ and taking a connected sum with $k$ copies of $\mathbb{CP}^2$: $X_k = X_0 \, \sharp \, k  \mathbb{CP}^2$.   Then $\partial X_k = M$, and by properties of the signature, $\tau(X_k) = \tau(X_0) + k$.  Therefore, if we take $k_0$ large enough,
\begin{align} \label{ab}
\tau(X_{k_0}) > \eta(M^3,[g]).
\end{align}

Next, Let $Y_{\ell}$ be the manifold obtained by deleting a ball in the interior of $X_{k_0}$ and taking a connected sum with $\ell$ copies of $S^2 \times S^2$.  Then $\partial Y_{\ell} = M$, and
\begin{align*}
\chi(Y_{\ell}) &= \chi(X_{k_0}) + 2 \ell, \\
\tau(Y_{\ell}) &= \tau(X_{k_0}).
\end{align*}
It follows that for all $\ell$ sufficiently large, say $\ell > \ell_0$,
\begin{align} \label{rev}
\tau(Y_{\ell}) = \tau(X_{k_0}) < \eta(M^3,[g]) + \frac{1}{3} \chi(Y_{\ell}),
\end{align}
while
\begin{align} \label{sob}
\tau(Y_{\ell}) = \tau(X_{k_0}) > \eta(M^3,[g]).
\end{align}
In particular, $Y_{\ell}$ satisfies the signature obstruction.

If $\ell > \ell_0$ and $Y_{\ell}$ admits a self-dual Poincar\'e-Einstein metric $g_{+}$ whose conformal infinity is given by $(M^3,[g])$,
then by Theorem \ref{SDThm} either $g_{+}$ is hyperbolic, or we have
\begin{align*}
\tau(Y_{\ell}) \geq \eta(M^3,[g]) + \frac{1}{3} \chi(Y_{\ell}).
\end{align*}
Since this contradicts (\ref{rev}), $g_{+}$ must be hyperbolic.  However, in this case the Atiyah-Patodi-Singer index formula implies $\tau(Y_{\ell}) = \eta(M^3,[g])$, but this contradicts (\ref{sob}).

 \end{proof}

\bigskip

\section{The proof of Theorem \ref{EvenThm}}  \label{EvenProof}

\begin{proof}[The proof of Theorem \ref{EvenThm}]   Let $(X^4, g_{+})$ be an even Poincar\'e-Einstein manifold.  Let $(M^3,[g])$ denote the conformal infinity, and assume $Y(M^3,[g]) > 0$.

Choose a representative (which we also denote $g$) in the conformal infinity $[g]$, and let $r > 0$ be the special defining function associated to $g$.  Let $\bg = r^2 g_{+}$.   The proof of Theorem \ref{EvenThm} is based on the following proposition, whose proof parallels the proof of Proposition \ref{signProp2}:

\begin{proposition}  \label{signProp1}  Under the assumptions of Theorem \ref{EvenThm},
\begin{align}  \label{1stYamabeEst}
Y_1(X,M,[\bg])^2 \leq  \frac{2}{3}    \int_X |W_{g_{+}}|^2 \, dv_{g_{+}}.
\end{align}
\end{proposition}

Before giving a proof of this proposition, let us show how Theorem \ref{EvenThm} follows.  First, as observed by J. Qing in \cite{Qing}, since $Y(M^3, [g]) > 0$ it follows that $Y_1(X,M,[\bg]) > 0$.  Also, the Fefferman-Graham expansion implies that $M$ is totally umbilic with respect to $\bg$.   Therefore, by Lemma \ref{YLemma} we have
\begin{align} \label{WZ1}
8 \pi^2 \chi(X) \leq  \int_X |W_{\bg}|^2 \, dv_{\bg} + \frac{3}{2} Y_1(X,M,[\bg])^2 = \int_X |W_{g_{+}}|^2 \, dv_{g_{+}}  + \frac{3}{2} Y_1(X,M,[\bg])^2.
\end{align}
Combining with (\ref{1stYamabeEst}), we obtain
\begin{align} \label{WZ2}
8 \pi^2 \chi(X) \leq  2 \int_X |W_{g_{+}}|^2 \, dv_{g_{+}}.
\end{align}
By Anderson's formula for the renormalized volume (\ref{AV2}),
\begin{align} \label{AV2}
V = \frac{4}{3} \pi^2 \chi(X) - \frac{1}{6} \int_X |W_{g_{+}}|^2 \, dv_{g_{+}},
\end{align}
which by (\ref{WZ2}) implies
\begin{align} \label{V1}
V \leq \frac{2}{3} \pi^2 \chi(X).
\end{align}

Also, by Lemma \ref{CGLemma} we have
\begin{align*}
 2 \sqrt{6} Y(M^3,[g])^{3/2} \leq Y_1(X,M[\bg])^2,
\end{align*}
hence by (\ref{1stYamabeEst})
\begin{align*}
2 \sqrt{6} Y(M^3,[g])^{3/2} \leq  \frac{2}{3}    \int_X |W_{g_{+}}|^2 \, dv_{g_{+}}.
\end{align*}
Again appealing to Anderson's formula (\ref{AV2}) we get
\begin{align} \label{V2}
V \leq \frac{4}{3} \pi^2 \chi(X) - \frac{\sqrt{6}}{2} Y(M^3,[g])^{3/2}.
\end{align}
Combining (\ref{V1}) and (\ref{V2}),
\begin{align*}
V \leq \min \left\{ \frac{2}{3} \pi^2 \chi(X),\frac{4}{3} \pi^2 \chi(X) - \frac{\sqrt{6}}{2} Y(M^3,[g])^{3/2} \right\}.
\end{align*}
\end{proof}

\begin{proof}[The proof of Proposition \ref{signProp1}]  Let $Z^{\pm} = r W^{\pm}_{g_{+}}$. By (\ref{WZ4}),
\begin{align} \label{WFbg2}
\frac{1}{2} \Delta_{\bg} |Z^{+}|^2 \geq |\nabla_{\bg} Z^{+}|^2   + \frac{1}{2} \left( R_{\bg} - 2 \sqrt{6} |W^{+}_{\bg}| \right) |Z^{+}|^2.
\end{align}
Let
\begin{align} \label{ZZ} \begin{split}
Z &= Z^{+} + Z^{-} \\
&= rW_{g_{+}}.
\end{split}
\end{align}
It follows that
\begin{align*}
|Z|^2 &= |Z^{+}|^2 + |Z^{-}|^2, \\
|\nabla_{\bg} Z|^2 &= |\nabla_{\bg} Z^{+}|^2 + |\nabla_{\bg} Z^{-}|^2,
\end{align*}
where the second follows since the connection preserves the splitting (\ref{lmp}).
Consequently, by (\ref{WFbg2}) we have
\begin{align} \label{WZ5}
\frac{1}{2} \Delta_{\bg} |Z|^2 = |\nabla_{\bg} Z|^2 -  \sqrt{6} |W^{+}_{\bg}| |Z^{+}|^2  -  \sqrt{6} |W^{-}_{\bg}| |Z^{-}|^2 + \frac{1}{2} R_{\bg} |Z|^2.
\end{align}
An elementary Lagrange-multiplier argument gives
\begin{align*}
\max_{\substack{x^2 + y^2 = 1, \\ a^2 + b^2 = 1}} \left( ax^2 + by^2 \right) = 1,
\end{align*}
hence
\begin{align} \label{WZ6}
-  \sqrt{6} |W^{+}_{\bg}| |Z^{+}|^2  -  \sqrt{6} |W^{-}_{\bg}| |Z^{-}|^2 \geq - \sqrt{6} |W_{\bg}| |Z|^2.
\end{align}
Substituting this into (\ref{WZ5}) gives
\begin{align} \label{WZ7}
\frac{1}{2} \Delta_{\bg} |Z|^2 &\geq |\nabla_{\bg} Z|^2 + \frac{1}{2} \left( R_{\bg} - 2\sqrt{6} |W_{\bg}|\right) |Z|^2.
\end{align}

Let $\epsilon > 0$, and define
\begin{align} \label{fez}
f_{\epsilon} = \left( \epsilon + |Z|^2 \right)^{1/6}.
\end{align}
Multiply (\ref{WZ7}) by $f_{\epsilon}^{-4}$, and integrate over $X$:
\begin{align} \label{WZ8}
\frac{1}{2} \int_X  f_{\epsilon}^{-4} \Delta_{\bg} |Z|^2  \, dv_{\bg} \geq \int_X f_{\epsilon}^{-4} |\nabla_{\bg} Z|^2 \, dv_{\bg}  + \frac{1}{2} \int_X  \left( R_{\bg} - 2\sqrt{6} |W_{\bg}|\right) f_{\epsilon}^{-4} |Z|^2 \, dv_{\bg}.
\end{align}
For the term on the left, we integrate by parts:
\begin{align} \label{WZ9} \begin{split}
\frac{1}{2} \int_X  f_{\epsilon}^{-4} \Delta_{\bg} |Z|^2  \, dv_{\bg} &= - \frac{1}{2} \int_X \langle \nabla_{\bg} \left( f_{\epsilon}^{-4} \right),  \nabla_{\bg} |Z|^2 \rangle_{\bg} \, dv_{\bg} + \frac{1}{2} \oint_M f_{\epsilon}^{-4} \frac{\partial }{\partial \nu} |Z|^2 \, dA_{g} \\
&= - \frac{1}{2} \int_X \langle \nabla_{\bg} \left( f_{\epsilon}^{-4} \right),  \nabla_{\bg} \left( f_{\epsilon}^6 \right) \rangle_{\bg} \, dv_{\bg} + \frac{1}{2} \oint_M f_{\epsilon}^{-4} \frac{\partial }{\partial \nu} |Z|^2 \, dA_{g} \\
&= 12 \int_X | \nabla_{\bg} f_{\epsilon} |^2 \, dv_{\bg} + \frac{1}{2} \oint_M f_{\epsilon}^{-4} \frac{\partial }{\partial \nu} |Z|^2 \, dA_g,
\end{split}
\end{align}
where $\nu$ is the outward normal to $M$ with respect to $g = \bg|_M$.

For the first term on the right-hand side of (\ref{WZ8}), we use the refined Kato inequality in (\ref{Kato}):  
\begin{align} \label{WZ10}
\int_X f_{\epsilon}^{-4} |\nabla_{\bg} Z|^2 \, dv_{\bg} &\geq \frac{5}{3} \int_X f_{\epsilon}^{-4} \left( |\nabla_{\bg} |Z^{+}||^2 + |\nabla_{\bg} |Z^{-}||^2 \right) \, dv_{\bg}.
\end{align}

We claim that (away from the zero loci of $|Z^{\pm}|$)
\begin{align} \label{WZ11}
\frac{5}{3} f_{\epsilon}^{-4} \left( |\nabla_{\bg} |Z^{+}||^2 + |\nabla_{\bg} |Z^{-}||^2 \right)   \geq 15   |\nabla_{\bg} f|^2.
\end{align}
To see this, compute
\begin{align*}
\nabla_{\bg} f &= \frac{1}{6} \left( \epsilon + |Z|^2 \right)^{-5/6} \left( 2 |Z^{+}| \nabla_{\bg} |Z^{+}| + 2 |Z^{-}| \nabla_{\bg} |Z^{-}| \right) \\
&= \frac{1}{3} f^{-5} \left( |Z^{+}| \nabla_{\bg} |Z^{+}| + |Z^{-}| \nabla_{\bg} |Z^{-}| \right);
\end{align*}
hence,
\begin{align*}
|\nabla_{\bg} f|^2 = \frac{1}{9} f^{-10} \Big(  |Z^{+}|^2 |\nabla_{\bg} |Z^{+}||^2 + 2 |Z^{+}||Z^{-}|\langle \nabla_{\bg} |Z^{+}|, \nabla_{\bg} |Z^{-}| \rangle  +  |Z^{-}|^2 |\nabla_{\bg} |Z^{-}||^2 \Big).
\end{align*}
Therefore,
\begin{align*}
\frac{5}{3} & f_{\epsilon}^{-4} \left( |\nabla_{\bg} |Z^{+}||^2 + |\nabla_{\bg} |Z^{-}||^2 \right) - 15   |\nabla_{\bg} f|^2 \\
&= \frac{5}{3} f_{\epsilon}^{-10} \Big\{ f_{\epsilon}^6 \left( |\nabla_{\bg} |Z^{+}||^2 + |\nabla_{\bg} |Z^{-}||^2 \right) - \Big(  |Z^{+}|^2 |\nabla_{\bg} |Z^{+}||^2 \\
& \ \ \ \  + 2 |Z^{+}||Z^{-}|\langle \nabla_{\bg} |Z^{+}|, \nabla_{\bg} |Z^{-}| \rangle  +  |Z^{-}|^2 |\nabla_{\bg} |Z^{-}||^2 \Big) \Big\} \\
&\geq \frac{5}{3} f_{\epsilon}^{-10} \Big\{ \left( |Z^{+}|^2 + |Z^{-}|^2 \right)   \left( |\nabla_{\bg} |Z^{+}||^2 + |\nabla_{\bg} |Z^{-}||^2 \right) - \Big(  |Z^{+}|^2 |\nabla_{\bg} |Z^{+}||^2 \\
& \ \ \ \  + 2 |Z^{+}||Z^{-}|\langle \nabla_{\bg} |Z^{+}|, \nabla_{\bg} |Z^{-}| \rangle  +  |Z^{-}|^2 |\nabla_{\bg} |Z^{-}||^2 \Big) \Big\} \\
&=  \frac{5}{3} f_{\epsilon}^{-10} \Big| |Z^{+} \nabla_{\bg} |Z^{-}| - |Z^{-}| \nabla_{\bg} |Z^{+}| \Big|^2 \\
&\geq 0,
\end{align*}
as claimed.

Substituting (\ref{WZ11}) into (\ref{WZ10}) and combining with (\ref{WZ8}) and (\ref{WZ9}) we obtain
\begin{align} \label{WZ13}
\frac{1}{2} \oint_M f_{\epsilon}^{-4} \frac{\partial }{\partial \nu} |Z|^2 \, dA_{\bg} \geq 3  \int_X |\nabla_{\bg} f_{\epsilon}|^2 \, dv_{\bg} + \frac{1}{2} \int_X \left( R_{\bg} - 2 \sqrt{6} |W_{\bg}| \right)  f_{\epsilon}^{-4} |Z|^2 \, dv_{\bg}.
\end{align}
We can rewrite the last term above as
\begin{align*}
\frac{1}{2} \int_X \left( R_{\bg} - 2 \sqrt{6} |W_{\bg}| \right)& f_{\epsilon}^{-4} |Z|^2 \, dv_{\bg} = \frac{1}{2} \int_X \left( R_{\bg} - 2 \sqrt{6} |W_{\bg}| \right) f_{\epsilon}^{-4} \left( f_{\epsilon}^6 - \epsilon \right) \, dv_{\bg} \\
&= \frac{1}{2} \int_X \left( R_{\bg} - 2 \sqrt{6} |W_{\bg}| \right) f_{\epsilon}^2 \, dv_{\bg} - \frac{1}{2} \epsilon \int_X \left( R_{\bg} - 2 \sqrt{6} |W_{\bg}| \right) f_{\epsilon}^{-2} \, dv_{\bg}.
\end{align*}
Since $f_{\epsilon} \geq \epsilon^{1/6}$,
\begin{align*}
- \frac{1}{2} \epsilon \int_X \left( R_{\bg} - 2 \sqrt{6} |W_{\bg}| \right) f_{\epsilon}^{-2} \, dv_{\bg} = O(\epsilon^{2/3}).
\end{align*}
Consequently, (\ref{WZ13}) implies
\begin{align} \label{WZ14} \begin{split}
3  \int_X |\nabla_{\bg} f_{\epsilon}|^2 \, dv_{\bg} + & \frac{1}{2} \int_X R_{\bg}   f_{\epsilon}^2 \, dv_g \\
&\leq \sqrt{6} \int_X |W_{\bg}|  f_{\epsilon}^2 \, dv_{\bg} + \frac{1}{2} \oint_M f_{\epsilon}^{-4} \frac{\partial }{\partial \nu} |Z|^2 \, dA_{\bg} + O(\epsilon^{2/3}).
\end{split}
\end{align}

It follows from the Fefferman-Graham expansions that the mean curvature of $M$ with respect to $g = \bg|_M$ vanishes.  Therefore, dividing (\ref{WZ14}) by $3$ we get
\begin{align} \label{WZ15}
\mathcal{E}_{\bg}[ f_{\epsilon} ] \leq \frac{\sqrt{6}}{3} \int_X |W_{\bg}|  f_{\epsilon}^2 \, dv_{\bg} +\frac{1}{6} \oint_M f_{\epsilon}^{-4} \frac{\partial }{\partial \nu} |Z|^2 \, dA_{\wg}  + O(\epsilon^{2/3}),
\end{align}
where $\mathcal{E}$ is defined in (\ref{E}).

To estimate the right-hand side of (\ref{WZ15}), we first observe that
\begin{align} \label{15p}
\frac{\sqrt{6}}{3} \int_X |W_{\bg}|  f_{\epsilon}^2 \, dv_{\bg} \leq \frac{\sqrt{6}}{3} \left( \int_X |W_{\bg}|^2 \, dv_{\bg} \right)^{1/2} \left(\int_X f_{\epsilon}^4 \, dv_{\bg} \right)^{1/2}
\end{align}

\begin{claim} \label{bClaim} The boundary integrand in (\ref{WZ15}) vanishes:
\begin{align} \label{NZ}
\frac{\partial }{\partial \nu} |Z|^2 = 0.
\end{align}
\end{claim}

Assuming the claim for now, it follows from (\ref{WZ15}) and (\ref{15p}) that
\begin{align} \label{WZ151} \begin{split}
\mathcal{E}_{\bg}[ f_{\epsilon} ] &\leq \frac{\sqrt{6}}{3}  \left( \int_X |W_{\bg}|^2 \, dv_{\bg} \right)^{1/2} \left(\int_X f_{\epsilon}^4 \, dv_{\bg} \right)^{1/2} + O(\epsilon^{2/3}) \\
&= \frac{\sqrt{6}}{3}  \left( \int_X |W_{g_{+}}|^2 \, dv_{g_{+}} \right)^{1/2} \left(\int_X f_{\epsilon}^4 \, dv_{\bg} \right)^{1/2} + O(\epsilon^{2/3}),
\end{split}
\end{align}
where the second line follows by the conformal invariance of the $L^2$-norm of the Weyl tensor.  By the definition of $Z$,
\begin{align*}
\int_X f_{\epsilon}^4 \, dv_{\bg} \geq \int_X |Z|^{4/3} \, dv_{\bg} = \int r^4 |W_{\bg}|_{\bg}^{4/3} \, dv_{\bg}.
\end{align*}
It follows that if $(X,g_{+})$ is not hyperbolic, then there is a $c_0 > 0$ such that
 \begin{align*}
\left( \int_X f_{\epsilon}^4 \, dv_{\bg} \right)^{1/2} \geq c_0 > 0,
\end{align*}
independent of $\epsilon$.  Dividing (\ref{WZ151}) by $\left( \int_X f_{\epsilon}^4 \, dv_{\bg} \right)^{1/2}$, we get
\begin{align} \label{WZ152}
\mathcal{Y}_{\bg}[ f_{\epsilon} ] \leq \frac{\sqrt{6}}{3}  \left( \int_X |W_{g_{+}}|^2 \, dv_{g_{+}} \right)^{1/2} + O(\epsilon^{2/3}),
\end{align}
where $\mathcal{Y}_{\bg}$ is the Yamabe functional associated to $\bg$ (see (\ref{YE}). Letting $\epsilon \to 0$, (\ref{1stYamabeEst}) follows.

\end{proof}

\begin{proof}[The proof of Claim \ref{bClaim}] By (\ref{ZZ}) and the conformal transformation law for the Weyl tensor,
\begin{align*}
|Z|^2 &= |Z|_{\bg}^2 \\
 &= \big| r W_{g_{+}} \big|_{\bg}^2 \\
&= r^2 \big| W_{g_{+}} \big|_{\bg}^2 \\
&= r^2 \big| r^{-2} W_{\bg} \big|_{\bg}^2 \\
&= r^{-2} | W_{\bg} |^2_{\bg}.
\end{align*}
By (\ref{W2g3}) of Corollary \ref{EvenCor},
\begin{align*}
|Z|^2 &= r^{-2} \big\{  r^2 |\mathcal{C} |^2 + O(r^4) \big\} \\
&= |\mathcal{C} |^2  + O(r^2),
\end{align*}
where $\mathcal{C}$ is with respect to $g = \bg|_M$.  Since $r$ is a special defining function, $\nu = - \frac{\partial}{\partial r}\vert_M$, hence
\begin{align} \label{dZM2} \begin{split}
\frac{\partial }{\partial \nu} |Z|^2 &= - \frac{\partial}{\partial r}  \big\{ |\mathcal{C}|^2  + O(r^2) \big\} \big|_{r=0}\\
&= 0,
\end{split}
\end{align}
as claimed.  \end{proof}

\bigskip

\section{Expansions for four-dimensional Poincar\'e-Einstein metrics}  \label{Expansions1}

In this section we calculate the expansion of the Weyl tensor that was needed in the proofs of the main results.  Let $(X^4, g_{+})$ be an oriented four-dimensional Poincar\'e-Einstein manifold with conformal infinity $(M,[\wg])$, where $M = \partial X$.  We fix a representative of the conformal infinity which we also denote by $\wg$, and let $r > 0$ denote a special defining function associated to $\wg$.  Let $\bg = r^2 g_{+}$.

In this section we closely follow the conventions and notation of Chapter 5 of \cite{FG}.   If $\{ \partial_1, \partial_2, \partial_3 \}$ is an oriented local basis of coordinate vector fields on $M$, then $\{ \partial_0, \partial_1, \partial_2, \partial_3 \}$ is an oriented local basis for a neighborhood in $X$.  We use lower case Greek indices when labeling components of tensor fields on $X$ (e.g., $0 \leq \alpha \leq 3$), and lower case Latin indices for tensor fields on $M$ (e.g., $1 \leq i \leq 3$).  The $0$-index always corresponds to $\partial_r$.

By the Fefferman-Graham expansions,
\begin{align} \label{gr}
g_{ij} = \wg_{ij} - r^2 \wP_{ij} + r^3 g^{(3)}_{ij} + r^4 g^{(4)}_{ij} + O(r^5),
\end{align}
where $\wP$ is the Schouten tensor with respect to $\wg$.  This implies the following expansions for the Christoffel symbols:
\begin{align} \label{bG} \begin{split}
\bar{\Gamma}_{ij}^k &= \widehat{\Gamma}_{ij}^k - \frac{1}{2} r^2 \big[ \wD_i \wP_j^k + \wD_j \wP_i^k - \wD^k \wP_{ij} \big] + O(r^3), \\
\bar{\Gamma}_{ij}^0 &= r \wP_{ij} - \frac{3}{2} r^2 g^{(3)}_{ij} - 2 r^3 g^{(4)}_{ij} + O(r^4), \\
\bar{\Gamma}_{i0}^k &= - r \wP_i^k + \frac{3}{2} r^2 (g^{(3)})_i^k + r^3 \big[ 2 (g^{(4)})_i^k - (\wP^2)_i^k \big] + O(r^4), \\
\bar{\Gamma}_{0j}^k &= \bar{\Gamma}_{00}^k = \bar{\Gamma}_{00}^0 = 0.
\end{split}
\end{align}
where $\wD$ is the connection with respect to $\wg$.  Using these formulas, and the formulas from page 48 of \cite{FG}, we conclude
\begin{align} \label{bR} \begin{split}
\bR_{ijk\ell} &= \widehat{R}_{ijk\ell} + \frac{1}{2}r^2 \Big[ \widehat{\nabla}_j C_{i k \ell} - \widehat{\nabla}_{i} C_{j k \ell} - \wg_{ik} (\wP^2)_{j \ell} + \wg_{i \ell} (\wP^2)_{j k}  \\
& \ \ \ \ \ + \wg_{j k} (\wP^2)_{i \ell} - \wg_{j \ell} (\wP^2)_{ik} - 4 \wP_{ik} \wP_{j \ell} + 4 \wP_{jk} \wP_{i\ell} \Big] + O(r^3), \\
\bR_{i 0 j 0} &= \wP_{ij} - 3r g^{(3)}_{ij} + r^2 \big[ (\wP^2)_{ij} - 6 g^{(4)}_{ij} \big] + O(r^3), \\
\bR_{0ijk} &= - r C_{ijk} + \frac{3}{2} r^2 \big[ \wD_k g^{(3)}_{ij} - \wD_j g^{(3)}_{ik} \big] + O(r^3),
\end{split}
\end{align}
where $C$ is the Cotton tensor with respect $\wg$:
\begin{align*}
C_{ijk} = \wD_k \wP_{ij} - \wD_j \wP_{ik}.
\end{align*}

\subsection{Expansion of the Weyl tensor}   To obtain expansions for the Weyl tensor, we first need expansions of the Schouten tensor $\overline{P}$ with respect to $\bg$.  Since $g_{+} = r^{-2} \bg$, the conformal transformation law for the Schouten tensor implies
\begin{align} \label{Pchange}
P_{g_{+}} = \bar{P} + \dfrac{\bar{\nabla}^2 r}{r} - \frac{1}{2} \dfrac{|\bar{\nabla} r|^2}{r^2} \bg.
\end{align}
By the Einstein condition and the fact that $r$ is a special defining function, this gives
\begin{align} \label{Pc}
\overline{P} =  -\dfrac{\bar{\nabla}^2 r}{r}.
\end{align}
Using the formulas for the Christoffel symbols in (\ref{bG}), we get
\begin{align} \label{bP}  \begin{split}
\overline{P}_{ij} &= \wP_{ij} - \frac{3}{2} r g^{(3)}_{ij} - 2 r^2 g^{(4)}_{ij} + O(r^3), \\
\overline{P}_{i0} &= 0,\\
\overline{P}_{00} &= 0.
\end{split}
\end{align}

By the standard decomposition of the curvature tensor,
\begin{align*}
\bW_{\alpha \beta \gamma \delta} = \bR_{\alpha \beta \gamma \delta} - \bg_{\alpha \gamma} \overline{P}_{\beta \delta} + \bar{g}_{\alpha \delta} \overline{P}_{\beta \gamma} + \bar{g}_{\beta \gamma} \overline{P}_{\alpha \delta} - \bar{g}_{\beta \delta} \overline{P}_{\alpha \gamma}.
\end{align*}
Using the formulas in (\ref{bR}) and (\ref{bP}), this implies
\begin{align} \label{bW} \begin{split}
\bW_{ijk\ell} &= \frac{3}{2} r \big[ \wg_{ik} g^{(3)}_{j\ell} - \wg_{i\ell} g^{(3)}_{jk} - \wg_{jk} g^{(3)}_{i \ell} + \wg_{j \ell} g^{(3)}_{ik} \big] + r^2 V_{ijk\ell} + O(r^3), \\
\bW_{i 0 j 0} &= -\frac{3}{2} r g^{(3)}_{ij} + r^2 \big[ (\wP)^2_{ij} - 4 g^{(4)}_{ij} \big] + O(r^3), \\
\bW_{0ijk} &= - r C_{ijk} + \frac{3}{2} r^2 \big[ \wD_k g^{(3)}_{ij} - \wD_j g^{(3)}_{ik} \big] + O(r^3),
\end{split}
\end{align}
where for notational convenience we define
\begin{align} \label{V4} \begin{split}
V_{ijk\ell} &= \frac{1}{2}\big( \wD_j C_{ik\ell} - \wD_i C_{jk\ell} \big) - \frac{1}{2} \wg_{ik} (\wP^2)_{j \ell} + \frac{1}{2} \wg_{i\ell} (\wP^2)_{jk} + \frac{1}{2} \wg_{jk} (\wP^2)_{i \ell} - \frac{1}{2} \wg_{j\ell} (\wP^2)_{ik}\\
& \ \ \ \ \ \ + 2 \big[ \wg_{ik} g^{(4)}_{j\ell} - \wg_{i\ell} g^{(4)}_{jk} - \wg_{jk} g^{(4)}_{i \ell} + \wg_{j \ell} g^{(4)}_{ik} \big].
\end{split}
\end{align}
Note that the formula for $\bW_{ijk\ell}$ implies that $V$ is a curvature-type tensor on $M$.

\smallskip
We can use the above expansions to explicitly compute $g^{(4)}$:

\begin{proposition} \label{g4Prop} Let $(X,g_{+})$ be a four-dimensional Poincar\'e-Einstein manifold, and $\wg$ a representative in the conformal infinity with associated special defining function $r$.  Then
\begin{align} \label{g4}
g^{(4)}_{ij} = \frac{1}{4} \Big( \wD^k C_{ijk} +  (\wP^2)_{ij} \Big).
\end{align}
\end{proposition}

\begin{proof} On $M$, we define the tensor
\begin{align} \label{V2sec} \begin{split}
V_{ij} &= \wg^{k \ell} V_{ikj\ell} \\
&= \frac{1}{2} \wD^k C_{ijk} - \frac{1}{2} (\wP^2)_{ij} - \frac{1}{2} |\wP|^2 \wg_{ij}  + 2 \big( g^{(4)}_{ij} + \mbox{tr}_{\wg} \, g^{(4)} \, \wg_{ij} \big).
\end{split}
\end{align}
Since $\bW$ and $g^{(3)}$ are trace-free,
\begin{align*}
0 &= \bg^{\alpha \beta} \bW_{i \alpha j \beta} \\
&= \bg^{00} \bW_{i0j0} + \bg^{k \ell} \bW_{ikj\ell} \\
&= \Big\{  -\frac{3}{2} r g^{(3)}_{ij} + r^2 \big[ (\wP)^2_{ij} - 4 g^{(4)}_{ij} \big]+ O(r^3) \Big\}   \\
& \ \ \ \ \ \ + \big\{ \wg^{k\ell} + O(r^2) \big\} \Big\{ \frac{3}{2} r \big[ \wg_{ij} g^{(3)}_{k\ell} - \wg_{i\ell} g^{(3)}_{jk} - \wg_{jk} g^{(3)}_{i \ell} + \wg_{k \ell} g^{(3)}_{ij} \big] + r^2 V_{ikj\ell} + O(r^3) \Big\} \\
&= r^2 \big[ \wP^2_{ij} - 4 g^{(4)}_{ij} + V_{ij} \big] + O(r^3),
\end{align*}
hence
\begin{align} \label{g4V}
4 g^{(4)}_{ij} = \wP^2_{ij} + V_{ij}.
\end{align}
Taking the trace gives
\begin{align*}
4 \mbox{tr}_{\wg} \, g^{(4)} &= |\wP|^2 + \mbox{tr}_{\wg} \, V \\
&= |\wP|^2 + \big( - 2 |\wP|^2  + 8  \mbox{tr}_{\wg} \, g^{(4)} \big) \\
&= - |\wP|^2  + 8  \mbox{tr}_{\wg} \, g^{(4)},
\end{align*}
so that
\begin{align} \label{tg4}
\mbox{tr}_{\wg} \, g^{(4)} = \frac{1}{4}  |\wP|^2.
\end{align}
Returning to (\ref{g4V}), we have
\begin{align} \label{g4V2} \begin{split}
4 g^{(4)}_{ij} &= \wP^2_{ij} + V_{ij} \\
&= \wP^2_{ij} + \frac{1}{2} \wD^k C_{ijk} - \frac{1}{2} (\wP^2)_{ij} - \frac{1}{2} |\wP|^2 \wg_{ij} + 2 \big( g^{(4)}_{ij} + \mbox{tr}_{\wg} \, g^{(4)} \, \wg_{ij} \big) \\
&= \wP^2_{ij} + \frac{1}{2} \wD^k C_{ijk} - \frac{1}{2} (\wP^2)_{ij} - \frac{1}{2} |\wP|^2 \wg_{ij}  + 2  g^{(4)}_{ij} + 2 \big(  \frac{1}{4} |\wP|^2   \big) \wg_{ij}  \\
&= \frac{1}{2} \wD^k C_{ijk} + \frac{1}{2} (\wP^2)_{ij} + 2  g^{(4)}_{ij},
\end{split}
\end{align}
and (\ref{g4}) follows.

\end{proof}

\medskip

\begin{lemma} \label{VLemma}  We have
\begin{align} \label{V4a}  
V_{ijk\ell} =  \wg_{ik}  \wD^m C_{j\ell m}  - \wg_{i \ell}  \wD^m C_{j k m}  - \wg_{jk}  \wD^m C_{i \ell m} + \wg_{j \ell}   \wD^m C_{i k m},
\end{align}
and
\begin{align} \label{V2a}
V_{ij} = \wD^k C_{ijk}.
\end{align}
\end{lemma}

\begin{proof} Note that (\ref{V2a}) is immediate from (\ref{g4V}) and Proposition \ref{g4Prop}.

Since $V_{ijk\ell}$ is a curvature-type tensor defined on a three-dimensional vector space, its fully trace-free part must vanish:
\begin{align*}
Z_{ijk\ell} &:= V_{ijk\ell} - \big( \wg_{ik} V_{j\ell} - \wg_{i\ell} V_{jk} - \wg_{jk} V_{i\ell} + \wg_{j \ell} V_{ik} \big) + \frac{1}{2}  \mbox{tr}_{\wg} \, V \, \big( \wg_{ik} \wg_{j\ell} - \wg_{i\ell} \wg_{jk} \big) \\
&=0.
\end{align*}
Then (\ref{V4a}) follows from this fact and (\ref{V2a}).
\end{proof}
We thus obtain the following.
\begin{proposition} \label{Wexp}  Let $(X,g_{+})$ be a four-dimensional Poincar\'e-Einstein manifold, and $\wg$ a representative in the conformal infinity with associated special defining function $r$.  Then
\begin{align} \label{bW2} \begin{split}
\bW_{ijk\ell} &= \frac{3}{2} r \big[ \wg_{ik} g^{(3)}_{j\ell} - \wg_{i\ell} g^{(3)}_{jk} - \wg_{jk} g^{(3)}_{i \ell} + \wg_{j \ell} g^{(3)}_{ik} \big] + r^2 V_{ijk\ell} + O(r^3), \\
\bW_{i 0 j 0} &= -\frac{3}{2} r g^{(3)}_{ij} -  r^2 V_{ij}  + O(r^3), \\
\bW_{0ijk} &= - r C_{ijk} + \frac{3}{2} r^2 \big[ \wD_k g^{(3)}_{ij} - \wD_j g^{(3)}_{ik} \big] + O(r^3),
\end{split}
\end{align}
where $V_{ijk\ell}$ and $V_{ij}$ are given by (\ref{V4a}) and (\ref{V2a}).
\end{proposition}

\medskip

\subsection{Expansions of $\bW^{\pm}$}  As in \cite{FG}, we denote the volume form of $g_r$ by $\mu$, and the volume form of $\bg$ by $\bm$.  Obviously
\begin{align*}
\bm_{ijk\ell} = 0,
\end{align*}
since the boundary is three-dimensional.  As pointed out in \cite{FG},
\begin{align} \label{mmb}
\bm_{0ijk} = \sqrt{\det g_r} \epsilon_{ijk}.
\end{align}

We begin with the expansion of $\bW^{+}_{ijk\ell}$ (i.e., all tangential components).  To do this, we first compute the expansion of $(\star \bW)_{ijk\ell}$:
\begin{align} \label{sWt} \begin{split}
(\star \bW)_{ijk\ell} &= \frac{1}{2} \bm_{ij}^{\ \ \rho \sigma} \bW_{\rho \sigma k \ell} \\
&= \frac{1}{2}\bm_{ij \nu \theta} \bg^{\nu \rho} \bg^{\theta \sigma} \bW_{\rho \sigma k \ell} \\
&=  \bm_{ijm0} \, \bg^{mp} \bW_{p0k\ell} \\
&= \sqrt{\det g_r} \, \bg^{mp}  \epsilon_{ijm} \bW_{0 pk\ell}.
\end{split}
\end{align}
From the expansions of $g_r$ above we know
\begin{align*}
\sqrt{\det g_r} \, \bg^{mp}\, \epsilon_{ijm} &= \big( 1 + O(r^2) \big) \big( \wg^{mp} + O(r^2) \big) \sqrt{\det \wg} \, \epsilon_{ijm}\\
&= \wg^{mp} \sqrt{\det \wg} \, \epsilon_{ijm} + O(r^2) \\
&= \hm_{ij}^{\ \ p} + O(r^2),
\end{align*}
where $\hm$ is the volume form of $\wg$.  Then using the expansion of $\bW$ in (\ref{bW}), we get
\begin{align} \label{stbWtan}
(\star \bW)_{ijk\ell} =   - r \, \hm_{ij}^{\ \ p} \, C_{pk\ell} + \frac{3}{2} r^2  \big( \wD_{\ell} g^{(3)}_{pk} - \wD_k g^{(3)}_{p \ell} \big) \hm_{ij}^{\ \ p} + O(r^3).
\end{align}
Combining with the first formula in (\ref{bW}) gives
\begin{align} \label{bWpt} \begin{split}
\bW^{+}_{ijk\ell} &= \frac{1}{2} \big( \bW_{ijk\ell} + (\star \bW)_{ijk\ell} \big) \\
&= r \Big\{ \frac{3}{4}  \big[ \wg_{ik} g^{(3)}_{j\ell} - \wg_{i\ell} g^{(3)}_{jk} - \wg_{jk} g^{(3)}_{i \ell} + \wg_{j \ell} g^{(3)}_{ik} \big]
-\frac{1}{2}  \, \hm_{ij}^{\ \ p} \, C_{pk\ell}  \Big\}  \\
& \ \ \ \ + r^2 \Big\{ \frac{1}{2}  V_{ijk\ell} + \frac{3}{4} \big( \wD_{\ell} g^{(3)}_{pk} - \wD_k g^{(3)}_{p \ell} \big) \hm_{ij}^{\ \ p}\Big\} + O(r^3).
\end{split}
\end{align}
Similarly,
\begin{align} \label{bWmt} \begin{split}
\bW^{-}_{ijk\ell} &=  \frac{1}{2} \big( \bW_{ijk\ell} - (\star \bW)_{ijk\ell} \big) \\
&= r \Big\{ \frac{3}{4}  \big[ \wg_{ik} g^{(3)}_{j\ell} - \wg_{i\ell} g^{(3)}_{jk} - \wg_{jk} g^{(3)}_{i \ell} + \wg_{j \ell} g^{(3)}_{ik} \big]
+\frac{1}{2}  \, \hm_{ij}^{\ \ p} \, C_{pk\ell}  \Big\}  \\
& \ \ \ \ + r^2 \Big\{ \frac{1}{2}  V_{ijk\ell} - \frac{3}{4} \big( \wD_{\ell} g^{(3)}_{pk} - \wD_k g^{(3)}_{p \ell} \big) \hm_{ij}^{\ \ p}\Big\} + O(r^3).
\end{split}
\end{align}

Next we compute $(\star \bW)_{i0j0}$:
\begin{align*}
(\star \bW)_{i0j0} &= \frac{1}{2} \bg^{km} \bg^{\ell p} \bm_{0ik\ell} \bW_{0jmp} \\
&= \frac{1}{2} \bg^{km} \bg^{\ell p} \sqrt{ \det g_r } \, \epsilon_{ik\ell} \bW_{0jmp}.
\end{align*}
Consequently,
\begin{align} \label{stbWtntn}
(\star \bW)_{i0j0} =   -\frac{1}{2} r \, \hm_i^{\ mp} \, C_{jmp}   + \frac{3}{4} r^2  \big( \wD_{p} g^{(3)}_{jm} - \wD_m g^{(3)}_{jp} \big)\hm_i^{\ mp} + O(r^3),
\end{align}
and
\begin{align} \label{bWptntn} \begin{split}
\bW^{+}_{i0j0} &= \frac{1}{2} \big( \bW_{i0j0} + (\star \bW)_{i0j0} \big) \\
&= r \Big\{ -\frac{3}{4}  g^{(3)}_{ij} - \frac{1}{4}  \,\hm_i^{\ mp} \, C_{jmp}    \Big\} \\
& \ \ + r^2 \Big\{ -\frac{1}{2}  V_{ij} + \frac{3}{8} \big( \wD_{p} g^{(3)}_{jm} - \wD_m g^{(3)}_{jp} \big) \hm_i^{\ mp} \Big\}  + O(r^3).
\end{split}
\end{align}
Likewise,
\begin{align} \label{bWmtntn} \begin{split}
\bW^{-}_{i0j0} &= r \Big\{ -\frac{3}{4}  g^{(3)}_{ij} + \frac{1}{4}  \,\hm_i^{\ mp} \, C_{jmp}  \Big\} \\
& \ \ + r^2 \Big\{ - \frac{1}{2} V_{ij}  - \frac{3}{8} \big( \wD_{p} g^{(3)}_{jm} - \wD_m g^{(3)}_{jp} \big) \hm_i^{\ mp}  \Big\}  + O(r^3).
\end{split}
\end{align}

Finally,
\begin{align*}
(\star \bW)_{0ijk} &= \frac{1}{2} \bg^{\ell p } \bg^{q m} \bm_{0i \ell m} \bW_{pq jk} \\
&= \frac{1}{2} \bg^{\ell p} \bg^{q m} \sqrt{ \det g_r } \, \epsilon_{i \ell m} \bW_{pqjk};
\end{align*}
hence
\begin{align*}
(\star \bW)_{0ijk} &= \frac{3}{4} r \Big\{ \wg_{jp} g^{(3)}_{kq} - \wg_{kp} g^{(3)}_{jq} - \wg_{jq} g^{(3)}_{kp} + \wg_{kq} g^{(3)}_{jp}  \Big\}\, \hm_i^{\ pq} + \frac{1}{2} r^2 \hm_i^{\ pq} V_{pq jk} + O(r^3). \\
\end{align*}
It follows that
\begin{align} \label{bWpnttt} \begin{split}
\bW^{+}_{0ijk} &= r \Bigg\{  \frac{3}{8} \Big[ \wg_{jp} g^{(3)}_{kq} - \wg_{kp} g^{(3)}_{jq} - \wg_{jq} g^{(3)}_{kp} + \wg_{kq} g^{(3)}_{jp}  \Big]\,  \hm_i^{\ pq} - \frac{1}{2}  C_{ijk} \Bigg\} \\
& \ \ \ \ + r^2 \Big\{ \frac{1}{4} \hm_i^{\ pq} V_{pqjk} + \frac{3}{4} \big[ \wD_k g^{(3)}_{ij} - \wD_j g^{(3)}_{ik} \big]\Big\}   + O(r^3),
\end{split}
\end{align}

\smallskip

\begin{align} \label{bWmnttt} \begin{split}
\bW^{-}_{0ijk} &= r \Bigg\{  - \frac{3}{8} \Big[ \wg_{jp} g^{(3)}_{kq} - \wg_{kp} g^{(3)}_{jq} - \wg_{jq} g^{(3)}_{kp} + \wg_{kq} g^{(3)}_{jp}  \Big]\,  \hm_i^{\ pq} - \frac{1}{2}  C_{ijk} \Bigg\}  \\
& \ \ \ \ + r^2 \Big\{ - \frac{1}{4} \hm_i^{\ pq} V_{pqjk} + \frac{3}{4} \big[ \wD_k g^{(3)}_{ij} - \wD_j g^{(3)}_{ik} \big]\Big\}   + O(r^3),
\end{split}
\end{align}

Summarizing, we have

\begin{proposition} \label{WpmProp}  Let $(X,g_{+})$ be an oriented, four-dimensional Poincar\'e-Einstein manifold.  Let $\wg$ be a representative of the conformal infinity, and $r$ the associated special defining function, and $\bg  = r^2 g_{+}$.  Then
\begin{align} \label{Wtable} \begin{split}
\bW^{\pm}_{ijk\ell} &= r \Big\{ \frac{3}{4}  \big[ \wg_{ik} g^{(3)}_{j\ell} - \wg_{i\ell} g^{(3)}_{jk} - \wg_{jk} g^{(3)}_{i \ell} + \wg_{j \ell} g^{(3)}_{ik} \big]
\mp \frac{1}{2}  \, \hm_{ij}^{\ \ p} \, C_{pk\ell}  \Big\} \\
& \ \ \ \ + r^2 \Big\{ \frac{1}{2}  V_{ijk\ell} \pm \frac{3}{4} \big( \wD_{\ell} g^{(3)}_{pk} - \wD_k g^{(3)}_{p \ell} \big) \hm_{ij}^{\ \ p}\Big\} + O(r^3),  \\
\bW^{\pm}_{i0j0} &= r \Big\{ -\frac{3}{4}  g^{(3)}_{ij} \mp \frac{1}{4}  \,\hm_i^{\ mp} \, C_{jmp}    \Big\} + r^2 \Big\{ -\frac{1}{2}  V_{ij} \pm \frac{3}{8} \big( \wD_{p} g^{(3)}_{jm} - \wD_m g^{(3)}_{jp} \big) \hm_i^{\ mp} \Big\}  + O(r^3), \\
\bW^{\pm}_{0ijk} &= r \Bigg\{ \pm \frac{3}{8} \Big[ \wg_{jp} g^{(3)}_{kq} - \wg_{kp} g^{(3)}_{jq} - \wg_{jq} g^{(3)}_{kp} + \wg_{kq} g^{(3)}_{jp}  \Big]\,  \hm_i^{\ pq} - \frac{1}{2}  C_{ijk} \Bigg\} \\
& \ \ \ \  + r^2 \Big\{ \pm \frac{1}{4} \hm_i^{\ pq} V_{pqjk} + \frac{3}{4} \big[ \wD_k g^{(3)}_{ij} - \wD_j g^{(3)}_{ik} \big]\Big\}   + O(r^3),
\end{split}
\end{align}
where the tensors $V_{ijk\ell}$ and $V_{ij}$ are given by (\ref{V4}) and (\ref{V2sec}).
\end{proposition}

\bigskip

\section{Expansions for even and self-dual metrics}  \label{Expansions2}

In this section we state two key corollaries of the preceding calculations.

\begin{corollary} \label{EvenCor}  Let $(X,g_{+})$ be an even, four-dimensional Poincar\'e-Einstein manifold, and $\wg$ a representative in the conformal infinity with associated special defining function $r$.
Then
\begin{align} \label{bW2g3} \begin{split}
\bW_{ijk\ell} &= r^2 V_{ijk\ell} + O(r^3), \\
\bW_{i 0 j 0} &= -  r^2 V_{ij}  + O(r^3), \\
\bW_{0ijk} &= - r C_{ijk} + O(r^3),
\end{split}
\end{align}
where $V_{ijk\ell}$ and $V_{ij}$ are given by (\ref{V4a}) and (\ref{V2a}).  In particular,
\begin{align} \label{W2g3}
|\bW|^2_{\bg} = r^2 |C|^2 + O(r^4).
\end{align}

Also, we have
\begin{align} \label{Wtableg3} \begin{split}
\bW^{\pm}_{ijk\ell} &= \mp \frac{1}{2} r \, \hm_{ij}^{\ \ p} \, C_{pk\ell} + \frac{1}{2} r^2  V_{ijk\ell} + O(r^3),  \\
\bW^{\pm}_{i0j0} &=  \mp \frac{1}{4} r \,\hm_i^{\ mp} \, C_{jmp}   -\frac{1}{2} r^2  V_{ij} + O(r^3), \\
\bW^{+}_{0ijk} &= - \frac{1}{2} r C_{ijk}  \pm \frac{1}{4} r^2 \hm_i^{\ pq} V_{pqjk} + O(r^3).
\end{split}
\end{align}
\end{corollary}

\begin{proof}  These formulas follow from Propositions \ref{Wexp} and \ref{WpmProp} and the fact that $g^{(3)} = 0$.  \end{proof}

\medskip


\begin{corollary} \label{SDCor3}  If $(X,g_{+})$ is an oriented, self-dual Poincar\'e-Einstein manifold, then
\begin{align} \label{bWSD}  \begin{split}
\bW^{+}_{ijk\ell} &= - r \, \hm_{ij}^{\ \ p} \, C_{pk\ell}  + r^2 V_{ijk\ell} + O(r^3), \\
\bW^{+}_{i0j0} &=  -\frac{1}{2} r \, \mathcal{C}_{ij}   - r^2  V_{ij}  + O(r^3),  \\
\bW^{+}_{0ijk} &= -  r C_{ijk} + \frac{1}{2} r^2  V_{pqjk} \, \hm_i^{\ pq}   + O(r^3),
\end{split}
\end{align}
where
\begin{align} \label{mC}
\mathcal{C}_{ij} = \hm_i^{\ k\ell} C_{j k \ell}.
\end{align}
\end{corollary}

\begin{proof}  Since $g_{+}$ is self-dual, it follows from (\ref{bWmt}) that
\begin{align} \label{g3formula} \begin{split}
&\frac{3}{4}  \big[ \wg_{ik} g^{(3)}_{j\ell} - \wg_{i\ell} g^{(3)}_{jk} - \wg_{jk} g^{(3)}_{i \ell} + \wg_{j \ell} g^{(3)}_{ik} \big] =
 - \frac{1}{2}  \, \hm_{ij}^{\ \ p} \, C_{pk\ell}, \\
& \ \ \ \ \ \ \ \  \frac{3}{4} \big( \wD_{\ell} g^{(3)}_{pk} - \wD_k g^{(3)}_{p \ell} \big) \hm_{ij}^{\ \ p} = \frac{1}{2}  V_{ijk\ell}.
\end{split}
\end{align}
Note that the first formula above is equivalent to (5.5) of Theorem 5.3 in \cite{FG}.

Substituting these into (\ref{bWpt}) we get the first formula in (\ref{bWSD}).  The proof of the other two formulas is similar.
\end{proof}

Determining the expansion of the norm of the Weyl tensor is much more involved in the self-dual case, so we state it as a separate result:

\begin{proposition} \label{WnormProp}  If $(X,g_{+})$ is an oriented self-dual Poincar\'e-Einstein manifold, then
\begin{align} \label{bWpnorm}
|\bW^{+}|_{\bg}^2 =  r^2 \, |\mathcal{C}|_{\wg}^2 + 4 r^3 \,  \langle V, \mathcal{C} \rangle + O(r^4).
\end{align}
\end{proposition}

\begin{proof}
First, we note that
\begin{align} \label{Weylnorm} \begin{split}
|\bW^{+}|_{\bg}^2 &= \frac{1}{4} \bg^{\alpha \mu} \bg^{\beta \nu} \bg^{\gamma \delta} \bg^{\sigma \tau} \bW^{+}_{\alpha \beta \gamma \sigma} \bW^{+}_{\mu \nu \delta \tau} \\
&= \frac{1}{4} \bg^{ip} \bg^{jq} \bg^{kr} \bg^{\ell s} \bW^{+}_{ijk\ell} \bW^{+}_{pq rs} + \bg^{ip} \bg^{jq} \bg^{kr}   \bW^{+}_{0ijk} \bW^{+}_{0 pqr} \\
& \ \ \ \   + \bg^{ip} \bg^{jq} \bW^{+}_{i 0 j 0} \bW^{+}_{p 0 q 0} \\
&:= I_1 + I_2 + I_3.
\end{split}
\end{align}
By the first formula in (\ref{bWSD}),
\begin{align} \label{I1a} \begin{split}
I_1 &= \frac{1}{4} \bg^{ip} \bg^{jq} \bg^{kr} \bg^{\ell s} \bW^{+}_{ijk\ell} \bW^{+}_{pq rs} \\
&= \frac{1}{4} \big( \wg^{ip} + O(r^2) \big)\big( \wg^{jq} + O(r^2) \big) \big(  \wg^{kr} + O(r^2) \big) \big(  \wg^{\ell s} + O(r^2) \big) \Big\{  - r \, \hm_{ij}^{\ \ m} \, C_{mk\ell}  + r^2 V_{ijk\ell} \\
& \ \ \  + O(r^3) \Big\} \Big\{    - r \, \hm_{pq}^{\ \ n} \, C_{nrs}  + r^2 V_{pqrs} + O(r^3) \Big\} \\
&= \frac{1}{4} r^2 \, \wg^{ip} \wg^{jq} \wg^{kr} \wg^{\ell s} \hm_{ij}^{\ \ m} \hm_{pq}^{\ \ n} C_{mk\ell} C_{nrs} - \frac{1}{2} r^3 \, \wg^{ip} \wg^{jq} \wg^{kr} \wg^{\ell s} \hm_{ij}^{\ \ m} \hm_{pq}^{\ \ n} C_{mk\ell} V_{pqrs} + O(r^4) \\
&= \frac{1}{4} r^2 \, \hm^{pqm} \hm_{pq}^{\ \ n} C_{mk\ell} C_n^{\ k \ell} - \frac{1}{2} r^3 \, \hm^{pqm} C_{m k \ell} V_{pq}^{\ \ k \ell} + O(r^4).
\end{split}
\end{align}
It will be convenient to rewrite the second term on the right.  First, by skew-symmetry of the volume form,
\begin{align} \label{I1b}
- \frac{1}{2} \, \hm^{pqm} C_{m k \ell} V_{pq}^{\ \ k \ell} = - \frac{1}{2} \, \hm^{mpq} V_{pq}^{\ \ k \ell} C_{m k \ell}.
\end{align}

\begin{claim} \label{VC}
\begin{align} \label{vceq}
- \frac{1}{2} \hm^{mpq} V_{pq}^{\ \ k \ell} C_{m k \ell} = \langle V, \mathcal{C} \rangle = V_i^j \mathcal{C}_j^i.
\end{align}
\end{claim}

\begin{proof}  In the following, we will repeatedly use the identity
\begin{align} \label{mid}
\hm^{pij} \hm_{pk\ell} = \delta_{ik} \delta_{j \ell} - \delta_{jk} \delta_{i \ell}.
\end{align}
From this, it follows that
\begin{align*}
 \frac{1}{2} \hm^m_{\ jk} \mathcal{C}_{mi} &= \frac{1}{2} \hm^m_{\ jk} \hm_m^{\ pq} C_{ipq} \\
 &= \frac{1}{2} \hm_{mjk} \hm^{mpq} C_{ipq} \\
 &= \frac{1}{2} \big( \delta_{jp} \delta_{kq} - \delta_{jq} \delta_{kp} \big) C_{ipq}  \\
 &= \frac{1}{2} \big( C_{ijk} - C_{ikj} \big) \\
 &= C_{ijk},
\end{align*}
which can also be expressed as
\begin{align} \label{CC}
C_{ijk} = \frac{1}{2} \hm_{mjk} \mathcal{C}_i^m.
\end{align}

Since $C_{[mk\ell]} = 0$, we can rewrite the term in (\ref{vceq}) as
\begin{align} \label{rw1} \begin{split}
- \frac{1}{2} \hm^{mpq} V_{pq}^{\ \ k \ell} C_{m k \ell} &= \frac{1}{2} \hm^{mpq} V_{pq}^{\ \ k \ell} C_{k\ell m}  + \frac{1}{2} \hm^{mpq} V_{pq}^{\ \ k \ell} C_{\ell m k }.
\end{split}
\end{align}
Hence, by (\ref{CC}),
\begin{align} \label{rw2} \begin{split}
- \frac{1}{2} \hm^{mpq} V_{pq}^{\ \ k \ell} C_{m k \ell} &= \frac{1}{4} \hm^{mpq} V_{pq}^{\ \ k \ell} \hm_{s\ell m} \mathcal{C}^s_k + \frac{1}{4} \hm^{mpq} V_{pq}^{\ \ k \ell} \hm_{smk} \mathcal{C}^s_{\ell} \\
&= - \frac{1}{4} \hm^{mpq} \hm_{m \ell s}  V_{pq}^{\ \ k \ell} \mathcal{C}^s_k - \frac{1}{4} \hm^{mpq} \hm_{msk} V_{pq}^{\ \ k \ell} \mathcal{C}^s_{\ell} \\
&= - \frac{1}{4} \big( \delta_{p \ell} \delta_{q s} - \delta_{\ell q} \delta_{ps} \big) V_{pq}^{\ \ k \ell} \mathcal{C}^s_k - \frac{1}{4}  \big( \delta_{ps} \delta_{qk} - \delta_{sq} \delta_{pk} \big) V_{pq}^{\ \ k \ell} \mathcal{C}^s_{\ell} \\
&=  V_s^k \mathcal{C}_k^s.
\end{split}
\end{align}

\end{proof}

For the first term in the last line of (\ref{I1a}), we use the identity (\ref{mid}) to show
\begin{align*}
\frac{1}{4} r^2 \, \hm^{pqm} \hm_{pq}^{\ \ n} C_{mk\ell} C_n^{\ k \ell} = \frac{1}{4} |\mathcal{C}|^2.
\end{align*}
Combining with Claim \ref{VC}, we get
\begin{align} \label{I1}
I_1 =  \frac{1}{4} r^2 |\mathcal{C}|^2 + r^3 \, \langle V , \mathcal{C} \rangle + O(r^4).
\end{align}

Using the second formula in (\ref{bWSD}) we find
\begin{align} \label{I2a} \begin{split}
I_2 &= \bg^{ip} \bg^{jq} \bg^{kr}   \bW^{+}_{0ijk} \bW^{+}_{0 pqr} \\
&= \big( \wg^{ip} + O(r^2) \big)\big( \wg^{jq} + O(r^2) \big) \big(  \wg^{kr} + O(r^2) \big) \Big\{ -  r C_{ijk} + \frac{1}{2} r^2  V_{pqjk} \, \hm_i^{\ pq}   + O(r^3) \Big\} \\
& \ \ \ \times \Big\{ -  r C_{pqr} + \frac{1}{2} r^2  V_{abqr} \, \hm_p^{\ ab}   + O(r^3) \Big\} \\
&= r^2 \, \wg^{ip} \wg^{jq} \wg^{kr}C_{ijk} C_{pqr} - r^3 \, \wg^{ip} \wg^{jq} \wg^{kr} C_{ijk} V_{abqr} \, \hm_p^{\ ab} + O(r^4) \\
&= \frac{1}{2}  r^2 |\mathcal{C}|^2 - r^3 \, \hm^{iab}  V_{ab}^{\ \ jk} C_{ijk}  + O(r^4).
\end{split}
\end{align}
By Claim \ref{VC}, this can be expressed as
\begin{align} \label{I2}
I_2 = \frac{1}{2} r^2 |\mathcal{C}|^2 + 2 r^3 \,  \langle V, \mathcal{C} \rangle  + O(r^4).
\end{align}

Finally,
\begin{align} \label{I3} \begin{split}
I_3 &= \bg^{ip} \bg^{jq} \bW^{+}_{i 0 j 0} \bW^{+}_{p 0 q 0} \\
&= \big( \wg^{ip} + O(r^2) \big) \big( \wg^{jq} + O(r^2) \big) \Big\{  -\frac{1}{2} r \, \mathcal{C}_{ij}   - r^2  V_{ij}  + O(r^3) \Big\} \Big\{ -\frac{1}{2} r \, \mathcal{C}_{pq}   - r^2  V_{pq}  + O(r^3) \Big\} \\
&= \frac{1}{4} r^2 \, |\mathcal{C}|_{\wg}^2 + r^3 \, \wg^{ip} \wg^{kq} \mathcal{C}_{ij} V_{pq} + O(r^4) \\
&= \frac{1}{4} r^2 \, |\mathcal{C}|_{\wg}^2 + r^3 \,  \langle V, \mathcal{C} \rangle + O(r^4).
\end{split}
\end{align}
Combining (\ref{I1}), (\ref{I2}), and (\ref{I3}), we get (\ref{bWpnorm}).
\end{proof}

\bigskip

\section{A conformal invariant in dimension three}  \label{CISec}


\begin{proof}[The proof of Theorem \ref{CIThm}]  Although it is possible to verify (\ref{changes}) directly, we will give a ``holographic'' construction of the invariants.

Let $(M,g)$ be a three-dimensional Riemannian manifold.  By Theorem 5.3 of \cite{FG}, there is an $\epsilon > 0$ and a (formal) self-dual Poincar\'e-Einstein metric $g_{+}$ defined on $X = M \times [0,\epsilon)$ whose conformal infinity is $(M,[g])$.  By ``formal", we mean that $g_{+}$ can be expressed as
\begin{align*}
g_{+} = r^{-2} \left( dr^2 + g_r \right),
\end{align*}
where $g_r$ is a one-parameter family of metrics on $M$ that is determined to infinite order.  In particular, if we write
\begin{align*}
g_r = g + g^{(2)} r^2 + g^{(3)} r^3 + g^{(4)} r^4 + \cdots,
\end{align*}
then the Einstein condition determines $g^{(2)}$ and $g^{(4)}$, while the self-duality condition determines $g^{(3)}$ (see (\ref{g3formula})).  Then, by the proof of Proposition \ref{WnormProp},
\begin{align} \label{bWpnorm2}
|\bW^{+}|_{\bg}^2 =  r^2 \, |\mathcal{C}_{g}|_{g}^2 + 4 r^3 \,  \langle V_g, \mathcal{C}_g \rangle_g + O(r^4),
\end{align}
where $\bg = r^2 g_{+} = dr^2 + g_r$.  Since
\begin{align*}
|W^{+}_{g_{+}}|^2_{g_{+}} = r^4 |\bW^{+}|_{\bg}^2,
\end{align*}
(\ref{bWpnorm2}) implies
\begin{align} \label{bWpnorm2a}
|W^{+}_{g_{+}}|^2_{g_{+}}  =  r^6 \, |\mathcal{C}_{g}|_{g}^2 + 4 r^7 \,  \langle V_g, \mathcal{C}_g \rangle_g + O(r^8).
\end{align}

Now, let $\tilde{g} \in [g]$, and write
\begin{align} \label{gug}
\tilde{g} = e^{2w_0} g
\end{align}
for some $w_0 \in C^{\infty}(M)$.  If we write $g_{+}$ in normal form with
\begin{align*}
g_{+} = \tilde{r}^{-2} \left( d\tilde{r}^2 + \tilde{g}_{\tilde{r}} \right),
\end{align*}
where $\tilde{g}_{\tilde{r}} \vert_{\tilde{r} = 0} = \tilde{g}$, then (\ref{bWpnorm2a}) becomes
\begin{align} \label{bWpnorm3a}
|W^{+}_{g_{+}}|^2_{g_{+}}  =  \tilde{r}^6 \, |\mathcal{C}_{\tilde{g}}|_{\tilde{g}}^2 + 4 \tilde{r}^7 \,  \langle V_{\tilde{g}}, \mathcal{C}_{\tilde{g}} \rangle_{\tilde{g}} + O(\tilde{r}^8).
\end{align}
By Lemma 2.2 of \cite{RobinRNV},
\begin{align*}
\tilde{r} = r e^w,
\end{align*}
where $w$ has an expansion that consists of only even powers of $r$:
\begin{align*}
w = w_0 + O(r^2),
\end{align*}
where $w_0$ is given in (\ref{gug}).  In particular,
\begin{align} \label{rrt} \begin{split}
\tilde{r}^6 &= r^6 e^{6w_0} + O(r^8),  \\
\tilde{r}^7 &= r^7 e^{7w_0} + O(r^9).
\end{split}
\end{align}
Substituting these into (\ref{bWpnorm3a}) and comparing with (\ref{bWpnorm2a}), we get (\ref{changes}).

\end{proof}

\bigskip

%
%

\end{document}